\documentclass[a4paper,12pt]{article}

\usepackage{amsthm}
\usepackage{latexsym}
\usepackage{dsfont}
\usepackage{bbm}
\usepackage{amssymb}
\usepackage{amsmath}
\usepackage{graphicx}

\numberwithin{equation}{section}

\theoremstyle{plain}
\newtheorem{Thm}{Theorem}[section]
\newtheorem{Lemma}[Thm]{Lemma}
\newtheorem{Cor}[Thm]{Corollary}
\newtheorem{Prop}[Thm]{Proposition}
\theoremstyle{remark}

\theoremstyle{definition}

\newtheorem{Def}[Thm]{Definition}

\DeclareMathOperator{\N}{\mathbb{N}}
\DeclareMathOperator{\Z}{\mathbb{Z}}
\DeclareMathOperator{\Q}{\mathbb{Q}}
\DeclareMathOperator{\R}{\mathbb{R}}

\DeclareMathOperator{\C}{\mathbb{C}}
\DeclareMathOperator{\V}{\mathbb{V}}

\DeclareMathOperator{\Prob}{\mathbb{P}}

\DeclareMathOperator{\E}{\mathbb{E}}

\DeclareMathOperator{\A}{\mathcal{A}}

\DeclareMathOperator{\Stable}{\mathcal{S}}
\DeclareMathOperator{\1}{\mathbbm{1}}

\DeclareMathOperator{\bT}{\mathbf{T}}
\DeclareMathOperator{\bC}{\mathbf{C}}
\DeclareMathOperator{\bX}{\mathbf{X}}
\DeclareMathOperator{\bL}{\mathbf{L}}

\renewcommand{\i}{\mathrm{i}\,}

\DeclareMathOperator{\Ku}{\mathit{K}}
\DeclareMathOperator{\Smooth}{\Upsilon}
\DeclareMathOperator{\Smoothn}{\Upsilon^{\mathit n}}

\DeclareMathOperator{\Fsum}{\mathcal{S}(\mathfrak{F})}
\DeclareMathOperator{\Surv}{\mathit{S}}

\DeclareMathOperator{\dx}{\mathrm{d} \mathit{x}}

\DeclareMathOperator{\dt}{\mathrm{d} \mathit{t}}

\title{Fixed points of the smoothing transform:	\\ Two-sided solutions}
\author{	Gerold Alsmeyer\footnote{ Gerold Alsmeyer,
        Institut f\"ur Mathematische Statistik,
        Universit\"at M\"unster,
        Einsteinstra\ss e 62,
        DE-48149 M\"unster,
        Germany}		\and
	 Matthias Meiners\footnote{Corresponding author: Matthias Meiners,
        Matematiska institutionen,
        Uppsala universitet,
        Box 480,
        751 06 Uppsala, Sweden.
        Email: matthias.meiners@math.uu.se.
        Research supported by DFG-grant Me 3625/1-1}}

\begin{document}

\thispagestyle{empty}
\maketitle

\begin{abstract}
Given a sequence $(C,T) = (C,T_1,T_2,\ldots)$ of real-valued random variables with $T_j \geq 0$ for all $j \geq 1$ and almost surely finite $N = \sup\{j \geq 1: T_j > 0\}$, the smoothing transform associated with $(C,T)$, defined on the set $\mathcal{P}(\R)$ of probability distributions on the real line, maps an element $P\in\mathcal{P}(\R)$ to the law of $C + \sum_{j \geq 1} T_j X_j$, where $X_1,X_2,\ldots$ is a sequence of i.i.d.\ random variables independent of $(C,T)$ and with distribution $P$. We study the fixed points of the smoothing transform, that is, the solutions to the stochastic fixed-point equation $X_{1}\stackrel{\mathrm{d}}{=}C + \sum_{j \geq 1} T_j X_j$. By drawing on recent work by the authors with J.D.\;Biggins, a full description of the set of solutions is provided under weak assumptions on the sequence $(C,T)$. This solves problems posed by Fill and Janson \cite{FJ2000} and Aldous and Bandyopadhyay \cite{AB2005}. Our results include precise characterizations of the sets of solutions to large classes of stochastic fixed-point equations that appear in the asymptotic analysis of divide-and-conquer algorithms, for instance the \texttt{Quicksort} equation.
\end{abstract}
\vspace{0,1cm}


% distributions (general theory)
\hphantom{2010 Mathematics Subject Classification:}
Secondary:		39B32,			% eqs. for complex functions
				60E10,				% characteristic functions
				60J80				% applications of BPs

\section{Introduction}	\label{sec:Intro}

Let $(C,T) = (C,T_1,T_2,\ldots)$ be a given sequence of real-valued random variables such that the $T_{j}$ are non-negative and
\begin{equation}	\label{eq:N<infty}
\Prob(N<\infty) ~=~ 1,
\end{equation}
where $N = \sup\{j \geq 1: T_j > 0\}$.
On the set $\mathcal{P}(\R)$ of probability distributions on the line, the smoothing transform $\Smooth$ (associated with $(C,T)$) is defined as the mapping
\begin{equation*}
\Smooth: \mathcal{P}(\R) \to \mathcal{P}(\R),	\quad
P\ \mapsto\ \mathcal{L}\left(C + \sum_{j \geq 1} T_j X_j\right),
\end{equation*}
where $X_1,X_2,\ldots$ is a sequence of i.i.d.\ random variables with common distribution $P$ and independent of $(C,T)$ and where $\mathcal{L}(X)$ denotes the law of a random variable $X$. A fixed point of this smoothing transform is given by any $P\in\mathcal{P}(\R)$ such that, if $X$ has distribution $P$, the equation
\begin{equation}	\label{eq:SumFP_inhom}
X ~\stackrel{\mathrm{d}}{=}~ C + \sum_{j \geq 1} T_j X_j
\end{equation}
holds true. We call this equation homogeneous if $C=0$, that is, if
\begin{equation}	\label{eq:SumFP}
X ~\stackrel{\mathrm{d}}{=}~ \sum_{j \geq 1} T_j X_j.
\end{equation}
On the set of \emph{non-negative} solutions to \eqref{eq:SumFP}, there is a substantial literature, \cite{Big1977,DL1983,Liu1998,BK1997,Iks2004,BK2005,AR2006,ABM2010}, and relatively complete results. Two-sided solutions to the homogeneous equation, with special focus on symmetric ones and those with finite variance, have been studied in \cite{Cal2003,CR2003} which also allow real-valued $T_j$, $j \geq 1$. Further, simultaneously to the development of this paper, Spitzmann \cite{Spi2010} solved the two-sided inhomogeneous equation under stronger assumptions on $(C,T)$.

Our approach to \eqref{eq:SumFP_inhom} and \eqref{eq:SumFP} is based on the use of characteristic functions. Indeed, \eqref{eq:SumFP_inhom} has an equivalent reformulation in terms of the characteristic function $\phi(t) := \E \exp(\i tX)$ of $X$ ($t\in\R$, $\i\!$ the imaginary unit), \textit{viz.}
\begin{equation}	\label{eq:SumFE_inhom}
\phi(t) ~=~ \E \exp(\i C t) \prod_{j \geq 1} \phi(T_j t)	\quad	(t \in \R).
\end{equation}
In the homogeneous case, this equation takes the form
\begin{equation}	\label{eq:SumFE}
\phi(t) ~=~ \E \prod_{j \geq 1} \phi(T_j t)	\quad	(t \in \R).
\end{equation}

Without loss of generality, we assume that $N$ satisfies
\begin{equation}
N	~=~	\sum_{j \geq 1} \1_{\{T_j > 0\}}
\end{equation}
and define the function
\begin{equation}	\label{eq:m}
m:[0,\infty)	\to	[0,\infty],
\quad
\theta \mapsto \E \sum_{j=1}^N T_j^{\theta}.
\end{equation}
$m$ plays a crucial role in the analysis of \eqref{eq:SumFP} and is the Laplace transform of the intensity measure of the point process
\begin{equation}	\label{eq:Z}
\mathcal{Z}	~:=~	\sum_{j=1}^{N} \delta_{S(j)},
\end{equation}
where $S(j) := -\log T_j$. Hence, $m$ is a convex and continuous function on the possibly unbounded interval $\{m < \infty\}$.

Throughout the paper, we make the following standing assumptions:
\begin{equation}	\tag{A1}	\label{eq:A1}
\Prob	\big(T_j \in \{0\} \cup r^{\Z}  \text{ for all } j \geq 1\big) < 1	\quad	\text{for all } r \geq 1.
\end{equation}
\begin{equation}	\tag{A2}	\label{eq:A2}
m(0)=\E N	>1.
\end{equation}
\begin{equation}	\tag{A3}	\label{eq:A3}
1=m(\alpha) < m(\beta)\text{ for some }\alpha>0 \text{ and all } \beta \in [0,\alpha).
\end{equation}
Condition \eqref{eq:A1} ensures that the point process $\mathcal{Z}$ is not concentrated on any lattice $\lambda \Z$, $\lambda > 0$, which is a natural assumption in view of examples of \eqref{eq:SumFP_inhom} coming from applications. As explained in Caliebe \cite{Cal2003}, only simple cases are ruled out when assuming \eqref{eq:A2}.
Moreover, in view of previous studies of \eqref{eq:SumFP} in more restrictive settings \cite{DL1983,Liu1998,AR2006,AM2010}, it is natural to make the assumption \eqref{eq:A3} on the behaviour of $m$. We refer to \cite[Theorem 6.1, Example 6.4]{AM2010} for the most recent discussion.
$\alpha$ will be called \emph{characteristic exponent (of $T$)}.
%In our theorems and at some other points, we list assumptions \eqref{eq:A1}-\eqref{eq:A3} again in order to provide easier access to these results. Nevertheless, we always assume validity of \eqref{eq:A1}-\eqref{eq:A3} except when explicitly discarded.

\section{Main results and applications}	\label{sec:results}

Let $\mathfrak{F}$ denote the set of characteristic functions of probability measures $\not = \delta_0$ on $\R$ and
\begin{equation}	\label{eq:F_Sigma(C)}
\Fsum(C)	~:=~	\{\phi \in \mathfrak{F}: \phi \text{ solves \eqref{eq:SumFE_inhom}}\}.
\end{equation}
(Note that $\Fsum(C)$ depends on $(C,T)$ but that only the dependence on $C$ is displayed because $T$ will be fixed in what follows.)
We use $\Fsum$ as an abbre\-viation for $\Fsum(0)$. Our aim is to provide a full description of the set $\Fsum(C)$. We begin with the homogeneous case $C=0$.

\subsection{The homogeneous case}	\label{subsec:hom_results}

In order to determine $\Fsum$, we need the following additional weak assumption.

\begin{equation}	\tag{A4}	\label{eq:A4}
\eqref{eq:A4a} \text{ or } \eqref{eq:A4b} \text{ holds},
\end{equation}
where
\begin{equation}	\tag{A4a}	\label{eq:A4a}
\E \sum_{j=1}^N T_j^{\alpha} \log T_j ~\in~ (-\infty,0)
\text{ and }
\E \Big(\sum_{j=1}^N T_j^{\alpha}\Big) \log^+ \Big(\sum_{j=1}^N T_j^{\alpha}\Big) ~<~ \infty;
\end{equation}
\begin{equation}	\tag{A4b}	\label{eq:A4b}
\text{There exists some } \theta \in [0,\alpha) \text{ satisfying }
m(\theta) ~<~ \infty.
\end{equation}
Indeed, \eqref{eq:A4} is enough to determine $\Fsum$ in the case that $\alpha \not = 1$. Before we proceed with the statement of our main results in the homogeneous case, we provide some background information on non-negative solutions to
\begin{equation}	\label{eq:FPE_W}
X	~\stackrel{\mathrm{d}}{=}~	\sum_{j \geq 1} T_j^{\alpha} X_j
\end{equation}
where $X_1,X_2,\ldots$ are i.i.d.\ copies of $X$ and independent of $T$. Solutions to \eqref{eq:FPE_W} play an important role because they appear as mixing distributions in all other cases.
Proposition 2.1 in \cite{ABM2010} states that there is a non-trivial (\textit{i.e.}\ non-zero), non-negative solution to \eqref{eq:FPE_W} and that its distribution is unique up to scaling. In what follows, we fix one such solution and denote it by $W$. Further information on $W$ will be provided in Subsection \ref{subsec:discussion}.

\begin{Thm}	\label{Thm:set_of_solutions_non-lattice}
Assume that \eqref{eq:A1}-\eqref{eq:A4} and $\alpha\in (0,2]\setminus\{1\}$ hold true. Then $\Fsum$ is given by the family
\begin{equation}	\label{eq:solutions_general_form_non-lattice_alpha_not=1}
\phi(t)	~=~
	\begin{cases}
	\E \exp\left(-\sigma^{\alpha} W |t|^\alpha
	\left[1-\i\beta \frac{t}{|t|}\tan\left(\frac{\pi\alpha}{2}\right)\right]\right),
	& \text{ if } \alpha \not=2, \\
	\E \exp(- \sigma^2 W t^2),
	& \text{ if } \alpha = 2.
	\end{cases}
\end{equation}
The range of the parameters is given by $\sigma > 0$, $\beta \in [-1,1]$ if $\alpha \not = 2$, and $\sigma > 0$ if $\alpha = 2$.
\end{Thm}

The case $\alpha = 1$ is more involved than the case $\alpha \not = 1$ due to a phenomenon called \emph{endogeny}, a notion coined by Aldous and Bandyopadhyay \cite{AB2005}. In order to determine $\Fsum$ in this case, we need one more assumption concerning $T$:
\begin{equation}	\tag{A5}	\label{eq:A5}
\E \sum_{j=1}^N T_j^{\alpha} (\log^- T_j)^2	~<~	\infty.
\end{equation}
Note that \eqref{eq:A5} is implied by \eqref{eq:A4b} whereas it constitutes a non-void assumption if \eqref{eq:A4a} holds but \eqref{eq:A4b} fails.

\begin{Thm}	\label{Thm:set_of_solutions_non-lattice_alpha=1}
Suppose that \eqref{eq:A1}-\eqref{eq:A5} and $\alpha = 1$ hold true.
Then $\Fsum$ is given by the family
\begin{equation}	\label{eq:solutions_general_form_non-lattice_alpha=1}
\phi(t)	~=~	\E \exp\left(\i \mu W t -\sigma W |t|\right),
\end{equation}
where $\mu \in \R$, $\sigma \geq 0$ and $(\mu,\sigma) \not = (0,0)$.
\end{Thm}

\subsection{The inhomogeneous case}	\label{subsec:inhom_results}

In order to solve the inhomogeneous equation, we do not only need assumptions on $T$ as before but also on $C$ as should not be surprising. Two important assumptions here are:
\begin{align}\label{C1}\tag{C1}
\begin{split}
		&\hspace{2cm}m(1) < \infty,\ \E |C| < \infty,\nonumber\\
		&\text{ and}\quad(\Smoothn(\delta_0))_{n \geq 0} \text{ is $\mathcal{L}^p$-bounded for some $p \geq 1$}.
		\end{split}
\end{align}
\begin{equation}\label{C2}\tag{C2}
m(\beta) < 1\text{ and }\E |C|^{\beta} < \infty\text{ for some }0 < \beta \leq 1.
\end{equation}

\begin{Thm}	\label{Thm:set_of_solutions_inhom_non-lattice}
Let $\Prob(C \not = 0) > 0$.
Suppose that \eqref{eq:A1}-\eqref{eq:A4} and one of the conditions \eqref{C1} and \eqref{C2} hold true. Additionally assume \eqref{eq:A5} in the case $\alpha = 1$. Then there exists a coupling $(W^*,W)$ of random variables such that $W^*$ solves \eqref{eq:SumFP_inhom}, $W$ is a non-trivial, non-negative solution to \eqref{eq:FPE_W} and the set of characteristic functions of solutions to \eqref{eq:SumFP_inhom} is given by the family
\begin{equation}	\label{eq:solutions_general_form_inhom_non-lattice}
\phi(t)	~=~
	\begin{cases}
	\E \exp\left(\i W^* t - \sigma^{\alpha} W |t|^\alpha
	\left[1-\i\beta \frac{t}{|t|}\tan\left(\frac{\pi\alpha}{2}\right)\right]\right),
	& \text{if } \alpha \not \in \{1,2\}, \\
	\E \exp\left(\i (W^*+\mu W) t -\sigma W|t|\right),
	& \text{if } \alpha = 1,	\\
	\E \exp(\i W^* t -\sigma^2 W t^2),
	& \text{if } \alpha = 2.
	\end{cases}
\end{equation}
The range of the parameters is given by $\sigma \geq 0$, $\beta \in [-1,1]$ if $\alpha \not \in \{1,2\}$, $\mu \in \R$, $\sigma \geq 0$ if $\alpha = 1$, and $\sigma \geq 0$ if $\alpha = 2$.
The coupling $(W^*,W)$ can be explicitly constructed in terms of the branching model introduced in Subsection \ref{subsec:WBM}: $W$ can be constructed via \eqref{eq:W<->varphi} and $W^*$ by taking the limit $n \to \infty$ in \eqref{eq:W*_n}.
\end{Thm}

\subsection{Applications}	\label{subsec:applications}

Examples of the stochastic fixed-point equation \eqref{eq:SumFP_inhom} and its homogeneous counterpart \eqref{eq:SumFP} abound in the asymptotic analysis of random recursive structures, see \textit{e.g.} \cite{NR2005} and \cite{AB2005} and the references therein. For their occurrence in stochastic geometry see \cite{PW2006} and the references therein. Here we confine ourselves to an explicit mention of a particularly prominent example of \eqref{eq:SumFP_inhom}, \textit{viz.}\ the \texttt{Quicksort} equation:
\begin{equation}	\label{eq:Quicksort}
X	~\stackrel{\mathrm{d}}{=}~	U X_1 + (1-U) X_2 + g(U)
\end{equation}
where $U \sim \mathrm{Unif}(0,1)$, $X_1,X_2$ are i.i.d.\ copies of $X$ independent of $U$, and
\begin{equation*}
g:(0,1) \to (0,1),	\quad	u \mapsto 2u \log u + 2 (1-u) \log(1-u) + 1.
\end{equation*}
This equation arises in the study of the asymptotic behaviour of the number $C_n$ of key comparisons \texttt{Quicksort} requires to sort a list of $n$ distinct reals, see \cite{Roe1991}. More precisely, $(C_n-\E C_n)/n$ converges weakly as $n \to \infty$ to a distribution $P$ on $\R$ which is a solution to \eqref{eq:Quicksort}. It is the unique solution with mean $0$ and finite variance. The set of all solutions to \eqref{eq:Quicksort} (without any moment constraints) has been determined by Fill and Janson \cite{FJ2000}. Their result is included in Theorem \ref{Thm:set_of_solutions_inhom_non-lattice} and stated next as a corollary.

\begin{Cor}	\label{Cor:Quicksort}
The set of characteristic functions $\phi$ of solutions $X$ to the \emph{\texttt{Quicksort}} equation \eqref{eq:Quicksort} is given by the family
\begin{equation*}
\phi(t)	~=~	\psi(t) \, \exp(\i \mu t - \sigma|t|), \quad	t \in \R
\end{equation*}
with $\mu \in \R$, $\sigma \geq 0$ and $\psi$ denoting the characteristic function of the distributional limit of $(C_n-\E C_n)/n$.
\end{Cor}

In other words, the set of solutions to \eqref{eq:Quicksort} equals the set
\begin{equation*}
\{P * \mathcal{C}(\mu,\sigma): \mu \in \R, \sigma \geq 0\}
\end{equation*}
where $P$ is the distribution pertaining to $\psi$ and $\mathcal{C}(\mu,\sigma)$ denotes the Cauchy distribution with parameters $\mu$ and $\sigma$. Here we interpret $\mathcal{C}(\mu,0)$ as the Dirac measure at $\mu$.

\subsection{Discussion of the main results}	\label{subsec:discussion}

We continue with a definition of stable distributions following \cite{ST1994}. We say that $Y$ has distribution $\Stable_{\alpha}(\sigma,\beta,\mu)$ for $\alpha \in (0,2]$, $\sigma \geq 0$, $\beta \in [-1,1]$ and $\mu \in \R$ if $Y$ has characteristic function $\exp(\psi_Y)$, where $\psi_Y(0) = 0$ and, for $t \not = 0$,
\begin{equation*}
\psi_Y(t)
~=~
\begin{cases}
\i\mu t-\sigma^\alpha |t|^\alpha \left(1-\i\beta \frac{t}{|t|}\tan\left(\frac{\pi\alpha}{2}\right)\right) & \text{ if } \alpha \neq 1, \\
\i\mu t-\sigma |t| \left(1 + \i\beta \frac{t}{|t|} \frac{2}{\pi} \log|t| \right)
& \text{ if } \alpha = 1.
\end{cases}
\end{equation*}
Here, $\alpha$ is called index of stability, $\sigma$ scale parameter, $\beta$ skewness parameter, and $\mu$ shift parameter. Notice that if $\alpha = 2$, $\beta$ becomes meaningless so that we can assume $\beta = 0$ in this case.
Now suppose that $Y$ is a random variable, defined on the same probability space as and independent of the pair $(W^*,W)$ in Theorem \ref{Thm:set_of_solutions_inhom_non-lattice}, with distribution $Y \sim \Stable_{\alpha}(\sigma,\beta,0)$ for $\alpha \in (0,2)\setminus\{1\}$, $\sigma > 0$ and $\beta \in [-1,1]$. Then a standard calculation shows that $W^* + W^{1/\alpha} Y$ has characteristic function $\phi$ as in \eqref{eq:solutions_general_form_inhom_non-lattice}. Thus Theorem \ref{Thm:set_of_solutions_inhom_non-lattice} implies that any solution $X$ to \eqref{eq:SumFP} has a representation of the form
\begin{equation}
X	~\stackrel{\mathrm{d}}{=}~	W^* + W^{1/\alpha} Y
\end{equation}
for an appropriate stable random variable $Y$. The same holds with $W^*=0$ in the homogeneous case.
Analogous constructions can be made in the cases $\alpha = 1$ and $\alpha = 2$. See also \cite{Rue2006} for the connection between fixed points of  the inhomogeneous and the corresponding homogeneous smoothing transform.

The random variables $W^*$ and $W$ have been studied in the literature.	If \eqref{eq:A1}-\eqref{eq:A3} and \eqref{eq:A4a} hold, then $W$ can be chosen as the intrinsic martingale limit of an appropriate branching random walk, see the beginning of Subsection \ref{subsec:endogeny_disintegration} for more details. This limit has been well studied in a host of articles concerning existence of moments or its tail behaviour, see \textit{e.g.}\ \cite{Iks2004,AK2005,JO2010b} to name but a few. The random variable $W^*$ can be constructed from the weighted branching process introduced in Subsection \ref{subsec:WBM} and there are also general results on the tails of $W^*$, see \textit{e.g.}\ \cite{JO2010b,JO2010c}.

We finish this section with an overview of the further organization of this work and an outline of the proof of our main results. The first step will be the formulation of a weighted branching process in Section \ref{sec:iterating} that allows the iteration of \eqref{eq:SumFP_inhom} and \eqref{eq:SumFP} on a fixed probability space. The following Section \ref{sec:homogeneous} is devoted to the solutions of the homogeneous equation \eqref{eq:SumFP}. The simple inclusion to verify there is that the functions $\phi$ defined in Theorems \ref{Thm:set_of_solutions_non-lattice} and \ref{Thm:set_of_solutions_non-lattice_alpha=1} are actually characteristic functions solving the functional equation \eqref{eq:SumFE}. This will be done in Subsection \ref{subsec:simple_inclusions}. The proof of the reverse inclusion is more involved and requires as the basic tool the use of multiplicative martingales derived from the characteristic functions of solutions to \eqref{eq:SumFP}, see Subsection \ref{subsec:disintegration}. The limits of these martingales are one-to-one with the solutions to the functional equation \eqref{eq:SumFE}. Further, they are stochastic processes that satisfy a pathwise counterpart of the functional equation \eqref{eq:SumFE}. Since, on the other hand, it is known from earlier work by Caliebe \cite{Cal2003} that the paths of the martingale limits are characteristic functions of infinitely divisible distributions and thus possess a unique L\'evy representation, one can deduce a pathwise equation for the random L\'evy exponent involved, see \eqref{eq:Psi's_equation}. Starting from this equation, which constitutes the heart of our approach, we determine the random L\'evy measures of the martingale limits in Subsection \ref{subsec:Levy_measure}. In Subsection \ref{subsec:Proofs_homogeneous}, we completely solve \eqref{eq:Psi's_equation}, which in turn immediately leads to a proof of our main Theorems \ref{Thm:set_of_solutions_non-lattice} and \ref{Thm:set_of_solutions_non-lattice_alpha=1}. But before we can solve \eqref{eq:Psi's_equation}, we have to deal with the phenomenon of \emph{endogeny}. Endogenous fixed points will be introduced in Subsection \ref{subsec:endogenous_FPs}. They are special solutions to \eqref{eq:SumFP} that can entirely be defined in terms of the underlying weighted branching process. Their appearance complicates the analysis of \eqref{eq:Psi's_equation}. Therefore, we first determine all endogenous solutions to \eqref{eq:SumFP} in Subsection \ref{subsec:endogeny_disintegration}. To accomplish this, we make use of results on the asymptotic behaviour of general branching processes provided in Subsection \ref{subsec:HS_norming}. Section \ref{sec:inhom} is devoted to the study of the inhomogeneous equation \eqref{eq:SumFP_inhom}. Using again multiplicative martingales, we show the existence of an explicit one-to-one correspondence between the solutions to the inhomogeneous equation and the corresponding homogeneous one. From this result, it is easy to deduce Theorem \ref{Thm:set_of_solutions_inhom_non-lattice}.

\section{Iterating the fixed-point equation}	\label{sec:iterating}

Iteration forms a natural tool in the study of a functional equation which, in the case of Eqs \eqref{eq:SumFP_inhom} and \eqref{eq:SumFP}, leads to a weighted branching model associated with the input variable $(C,T)$. This model will be introduced next. It is intimately connected to the branching random walk based on the point process $\mathcal{Z}$ as introduced in \eqref{eq:Z}. We will discuss the connection to this branching random walk and general branching processes in Subsection \ref{subsec:BRW}.

\subsection{The weighted branching model}	\label{subsec:WBM}

Let $\V := \bigcup_{n \in \N_0} \N^n$, where $\N := \{1,2,\ldots\}$ and $\N^0 = \{\varnothing\}$. The elements $v \in \V$ will be called individuals or vertices. We abbreviate $v = (v_1,\ldots,v_n)$ by $v_1 \ldots v_n$ and write $v|k$ for the restriction of $v$ to the first $k$ entries, \textit{i.e.}, $v|k := v_1 \ldots v_k$ if $k \leq n$ and $v|k := v$ if $|v|>n$. Let $vw$ denote the vertex $v_1 \ldots v_n w_1 \ldots w_m$ where $w = w_1 \ldots w_m$. In this case, we say that $v$ is an ancestor of $vw$. The length of a node $v$ is denoted by $|v|$, thus $|v|=n$ iff $v\in\N^{n}$. Now let $\bC\otimes\bT := ((C(v),T(v)))_{v \in \V}$ be a family of i.i.d.\ copies of $(C,T)$, where $(C(\varnothing),T(\varnothing))=(C,T)$. We refer to $(C,T)=(C,T_{1},T_{2},\ldots)$ as the \emph{basic sequence (of the weighted branching model)} and interpret $C(v)$ as a weight attached to the vertex $v$ and $T_i(v)$ as a weight attached to the edge $(v,vi)$ in the infinite tree $\V$. Then define $L(\varnothing) := 1$ and, recursively, $L(vi) := L(v) T_i(v)$ for $v \in \V$ and $i \in \N$. For $n\in\N_0$, let $\A_n$ denote the $\sigma$-algebra generated by the $(C(v),T(v))$, $|v|<n$. Put also $\A_{\infty}:=\sigma(\A_n:n \geq 0)$ $=\sigma(\bC\otimes\bT)$.

Further, we assume the existence of a family $\bX = (X(v))_{v \in \V}$ of i.i.d.\ copies of $X$ which is independent of $\bC \otimes \bT$. Then $n$fold iteration of \eqref{eq:SumFP_inhom} can be expressed in terms of the weighted branching model:
\begin{equation}	\label{eq:SumFP_inhom_iterated}
X	~\stackrel{\mathrm{d}}{=}~	\sum_{|u|<n} L(u)C(u) + \sum_{|v|=n} L(v) X(v).
\end{equation}
(Notice that, almost surely, there are only finitely many non-zero terms in the sums above since we assume $\Prob(N<\infty)=1$.)
In the homogeneous case, the first sum on the right-hand side vanishes and \eqref{eq:SumFP_inhom_iterated} simplifies to
\begin{equation}	\label{eq:SumFP_iterated}
X	~\stackrel{\mathrm{d}}{=}~	\sum_{|v|=n} L(v) X(v).
\end{equation}
The functional equations \eqref{eq:SumFE_inhom} and \eqref{eq:SumFE} after $n$ iterations become
\begin{equation}	\label{eq:SumFE_inhom_iterated}
\phi(t) ~=~ \E \Bigg[ \exp\Bigg(\i t \sum_{|u|<n} L(u)C(u) \Bigg)
\, \prod_{|v|=n} \phi(L(v) t)\Bigg]	\qquad	(t \in \R)
\end{equation}
and
\begin{equation}	\label{eq:SumFE_iterated}
\phi(t) ~=~ \E \prod_{|v|=n} \phi(L(v) t)	\qquad	(t \in \R),
\end{equation}
respectively.

We close this subsection with the definition of the shift operators $[\cdot]_u$, $u \in \V$. Given any function $\Psi=\psi(\bC \otimes \bT)$ of the weight family $\bC \otimes \bT$ pertaining to $\V$, let $[\Psi]_u\ :=\ \psi(((C(uv),T(uv)))_{v \in \V})$ be the very same function but for the weight ensemble pertaining to the subtree rooted at $u \in \V$. Any branch weight $L(v)$ can be viewed as such a function, and we thus have $[L(v)]_u = T_{v_1}(u) \cdot ... \cdot T_{v_n}(uv_1 ... v_{n-1})$ if $v = v_1 ... v_n$. Hence if $L(u) > 0$, then $[L(v)]_u = L(uv)/L(u)$.

\subsection{The corresponding branching random walk}	\label{subsec:BRW}

The weighted branching model introduced in Subsection \ref{subsec:WBM} turns into a classical \emph{branching random walk (BRW)} model after logarithmic scaling. Define
\begin{equation}	\label{eq:S(v)}
S(v)	~:=~	-\log L(v),	\quad	v \in \V
\end{equation}
where $-\log 0 := \infty$ is stipulated. Further, let
\begin{equation}
\mathcal{Z}	~:=~	\sum_{j=1}^N \delta_{S(j)}
\end{equation}
and
\begin{equation}
\mathcal{Z}_n	~:=~	\underset{S(v)<\infty}{\sum_{|v|=n:}} \delta_{S(v)},
\quad	n \in \N_0.
\end{equation}
Then $(\mathcal{Z}_n)_{n \geq 0}$ forms a classical BRW based on the point process $\mathcal{Z}$. BRWs have been studied in many articles, see \textit{e.g.}\ \cite{Big1977,BK1997,Big1998,HS2009} and the references therein.

We continue with a collection of some known facts about BRWs which will be useful in the course of the proof of our main results.

Let $\Surv$ denote the set of survival of the branching process:
\begin{equation}	\label{eq:Surv}
\Surv	~:=~	\bigg\{\sup_{|v|=n} L(v) > 0 \text{ for all } n \geq 0\bigg\}.
\end{equation}
The supercriticality assumption \eqref{eq:A2} guarantees that $\Prob(\Surv) > 0$.

Our approach to understanding \eqref{eq:SumFP} is based on the analysis of its iterated version \eqref{eq:SumFP_iterated}. To understand the latter equation, we need input on the asymptotic behaviour of the weights $L(v)$ or, equivalently, the positions $S(v)$.

\begin{Lemma}[Theorem 3 in \cite{Big1998}]	\label{Lem:sup_L(v)_to_0}
Under \eqref{eq:A1}-\eqref{eq:A3},
$B_n := \inf_{|v|=n} S(v) \to \infty$ almost surely on $\Surv$. In particular,
\begin{equation}	\label{eq:sup_L(v)_to_0}
\sup_{|v|=n} L(v)	~=~	e^{-B_n}	~\to~	0
\quad	\text{almost surely as } n \to \infty.
\end{equation}
\end{Lemma}

\subsection{The embedded BRW with positive steps only}	\label{subsec:embedded}

In some arguments involving the BRW $(\mathcal{Z}_n)_{n \geq 0}$, it is convenient to consider an embedded BRW with positive steps only. This idea is not new, see \textit{e.g.}\ \cite[Section 3]{BK2005}. Thus, we keep the construction of the embedded BRW short.

Let $\mathcal{G}_n := \{v \in \N^n: S(v) < \infty\} = \{v \in \N^n: L(v)>0\}$. The $\mathcal{G}_n$ are the generations of the original BRW. $\mathcal{G} := \bigcup_{n \geq 0} \mathcal{G}_n$ is the set of all population members of the original BRW. Now let $\mathcal{G}^>_0 := \{\varnothing\}$, and, recursively, for $n \geq 1$,
\begin{equation*}	%	\label{eq:G_n^>}
\mathcal{G}_n^> := \{vw \in \mathcal{G}: v \in \mathcal{G}^>_{n-1}, S(vw)\!>\!S(v), S(vw|k) \leq S(v) \;\text{for}\; |v| \leq k < |vw|\}.
\end{equation*}
The sequence $(\mathcal{G}^>_{n})_{n \geq 0}$ is an embedded generation sequence that contains exactly those individuals $v$ the positions of which are strict records in the random walk $S(\varnothing), S(v|1), \ldots, S(v)$. Using the $\mathcal{G}^>_n$, we can define the $n$th generation point process of the embedded BRW of strictly increasing ladder heights by
\begin{equation}	\label{eq:Z^>_n}
\mathcal{Z}^>_n	~:=~	\sum_{v \in \mathcal{G}^>_n} \delta_{S(v)}.
\end{equation}
$(\mathcal{Z}^>_n)_{n \geq 0}$ is again a BRW but with positive steps only. The following result states that the assumptions \eqref{eq:A1}-\eqref{eq:A5}, which can be interpreted as assumptions on the point process $\mathcal{Z}$, are passed on to the point process $\mathcal{Z}^> := \mathcal{Z}^>_1$:

\begin{Prop}	\label{Prop:embedded_BRW}
Assuming \eqref{eq:A1}-\eqref{eq:A3}, the following assertions hold:
\begin{itemize}
	\item[(a)]
		$\Prob(|\mathcal{G}^>_1|< \infty) = 1$.
	\item[(b)]
		$\mathcal{Z}^>$ satisfies \eqref{eq:A1}-\eqref{eq:A3} where \eqref{eq:A3} holds with the same $\alpha$ as for $\mathcal{Z}$.
	\item[(c)]
		If $\mathcal{Z}$ further satisfies \eqref{eq:A4a} or \eqref{eq:A4b},
		then the same holds true for $\mathcal{Z}^>$, respectively.
	\item[(d)]
		If $\mathcal{Z}$ satisfies \eqref{eq:A4} and \eqref{eq:A5}, then so does $\mathcal{Z}^>$.
\end{itemize}
\end{Prop}
\begin{proof}
Assertion (a) follows from \cite[Theorem 10(d)]{BK2005}. Assertions (b) and (c) follow from \cite[Lemma 9.1]{ABM2010}. It remains to prove that, given \eqref{eq:A1}-\eqref{eq:A5}, \eqref{eq:A5} holds for $\mathcal{Z}^>$ as well, in other words, that
\begin{equation*}
\E \sum_{j=1}^N e^{-\alpha S(j)} (S(j)^+)^2 ~<~	\infty
\quad	\text{implies}	\quad
\E \sum_{v \in \mathcal{G}^>_1} e^{-\alpha S(v)} (S(v)^+)^2 ~<~	\infty.
\end{equation*}
By what has already been shown, if $\mathcal{Z}$ satisfies \eqref{eq:A4b}, then so does $\mathcal{Z}^>$. Then, since \eqref{eq:A4b} implies \eqref{eq:A5}, the latter condition also holds for $\mathcal{Z}^>$. Therefore, we may assume that $\mathcal{Z}$ satisfies \eqref{eq:A4a} but not necessarily \eqref{eq:A4b}. Let $(S_n)_{n \geq 0}$ be a standard random walk with increment distribution $\mu_{\alpha} = \E \sum_{j=1}^N e^{-\alpha S(j)} \delta_{S(j)}$ (notice that \eqref{eq:A3} renders $\mu_{\alpha}$ a probability distribution on $\R$). From (7.2) in \cite{ABM2010}, we infer that $\mu_{\alpha}^> := \E \sum_{v \in \mathcal{G}_1^>} e^{-\alpha S(v)} \delta_{S(v)}$ is the distribution of the first ladder height of the random walk $(S_n)_{n \geq 0}$, that is, $\mu_{\alpha}^>(\cdot)=\Prob(S_{\sigma} \in \cdot)$ where $\sigma := \inf\{n \geq 0: S_n > 0\}$.
Now \eqref{eq:A5} for $\mathcal{Z}$ can be restated as $\E (S_1^+)^2 < \infty$, whereas \eqref{eq:A5} for $\mathcal{Z}^>$ means that $\E S_\sigma^2 < \infty$. But $\E (S_1^+)^2 < \infty$ and $\E S_\sigma^2 < \infty$ are equivalent, for $\E S_1 = -m'(\alpha) \in (0,\infty)$, see Theorem 3.1 in \cite{Gut2009}.
\end{proof}

\section{Solving the homogeneous equation}	\label{sec:homogeneous}

\subsection{The simple inclusions}	\label{subsec:simple_inclusions}

We begin our analysis of the homogeneous equation by verifying the simple inclusions in our main results. To be more precise, we prove in this subsection that any $\phi$ as defined in \eqref{eq:solutions_general_form_non-lattice_alpha_not=1} or \eqref{eq:solutions_general_form_non-lattice_alpha=1} is an element of $\Fsum$.

\begin{proof}[Proof of Theorems \ref{Thm:set_of_solutions_non-lattice} and \ref{Thm:set_of_solutions_non-lattice_alpha=1}: The simple inclusions]	$\hphantom{NICHTS}$	\\
Recall that $W$ denotes a fixed non-trivial non-negative random variable satisfying \eqref{eq:FPE_W}. 
Assume that $\alpha \in (0,2) \setminus \{1\}$. Choose any $\sigma \geq 0$ and $\beta \in [-1,1]$ and assume that $\phi$ is given as in \eqref{eq:solutions_general_form_non-lattice_alpha_not=1}. As explained in Subsection \ref{subsec:discussion} it follows that $\phi$ is the characteristic function of $W^{1/\alpha} Y$ for some random variable $Y \sim \Stable_{\alpha}(\sigma,\beta,0)$ which is independent of $W$. In particular, $\phi \in \mathfrak{F}$. Thus, it remains to show that $\phi$ solves \eqref{eq:SumFE}. This can be done by a calculation in the spirit of Section 4 in \cite{ABM2010}.
%To this end, denote by $\Psi$ the moment generating function of $W$, that is, $\Psi(z) = \E \exp(z W)$ for all $z \in \C$ with $\E |\exp(zW)| = \E \exp(\mathrm{Re}(z) W) < \infty$. Note that the latter condition holds true on the halfplane $\{\mathrm{Re}(z) \leq 0\}$, for $W$ is non-negative. Denote by $W_1, W_2, \ldots$ a sequence of i.i.d.\ copies of $W$ independent of $T$. Then
%\begin{eqnarray*}
%\phi(t)
%& = &
%\E \Psi\left(-\sigma^{\alpha} |t|^\alpha
%\left[1-\i\beta \frac{t}{|t|}\tan\left(\frac{\pi\alpha}{2}\right)\right]\right).
%\end{eqnarray*}
%On the other hand,
%\begin{eqnarray*}
%\phi(t)
%& = &
%\E \exp\left(-\sigma^{\alpha} W |t|^\alpha
%\left[1-\i\beta \frac{t}{|t|}\tan\left(\frac{\pi\alpha}{2}\right)\right]\right)	\\
%& = &
%\E \exp\left(-\sum_{j \geq 1} \sigma^{\alpha} (T_j |t|)^\alpha W_j 
%\left[1-\i\beta \frac{t}{|t|}\tan\left(\frac{\pi\alpha}{2}\right)\right]\right)	\\
%& = &
%\E \prod_{j=1}^N \Psi \left(-\sigma^{\alpha} (|T_j t|)^\alpha  \left[1-\i\beta \frac{T_j t}{|T_j t|}\tan\left(\frac{\pi\alpha}{2}\right)\right] \right)	\\
%& = &
%\E \prod_{j \geq 1} \phi(tT_j).
%\end{eqnarray*}
Similar arguments apply when $\alpha \in \{1,2\}$.
\end{proof}

\subsection{Disintegration}	\label{subsec:disintegration}

For $\phi \in \Fsum$, define
\begin{equation} \label{eq:disintegrated}
\Phi_n(t)	~:=~	\Phi_n(t,\bL)	~:=~ \prod_{|v|=n} \phi(L(v)t),
\quad n \geq 0.
\end{equation}
Caliebe \cite{Cal2003} proved that, as $n \to \infty$, almost all paths of $\Phi_n$ tend to characteristic functions of infinitely divisible distributions. Since this result is of major importance for our further analysis, we will state it here in a form adapted to our notation. Recall that a measure $\nu$ on the Borel sets of $\R^* := \R \setminus \{0\}$ is called a \emph{L\'evy measure} if
\begin{equation}	\label{eq:Levy_measure}
\int \frac{x^2}{1+x^2} \, \nu(\dx)	~<~	\infty.
\end{equation}

\begin{Prop}	\label{Prop:Disintegration}
Let $ \phi \in \Fsum$. Then, almost surely as $n \to \infty$, $(\Phi_n)_{n \geq 0}$ converges pointwise to a random characteristic function $\Phi$ of the form $\Phi = \exp(\Psi)$ with
\begin{equation}	\label{eq:char_exponent}
\Psi(t)	~=~	\i W_1t - \frac{W_2 t^2}{2}
+ \int \left(e^{\i tx} - 1 - \frac{\i tx}{1+x^2} \right) \, \nu(\dx),
\quad	t \in \R,
\end{equation}
where $W_1$ and $W_2$ are $\R$- and $[0,\infty)$-valued $\bL$-measurable random variables, respectively, and $\nu$ is a (random) L\'evy measure such that, for any $t > 0$, $\nu([t,\infty))$ and $\nu((-\infty,-t])$ are $\bL$-measurable. Moreover,
\begin{equation}	\label{eq:disintegration_integrated}
\E \Phi(t) ~=~ \phi(t)	\quad	\text{for all } t \in \R.
\end{equation}
\end{Prop}
\begin{proof}
Except for the measurability statements, this is a reformulation of Theorem 1 in \cite{Cal2003}. The measurability assertions can partly be concluded from \cite[Lemma 6]{Cal2003}. 
For ease of reference here and later, this lemma is stated next.

\begin{Lemma}[Lemma 6 in \cite{Cal2003}]	\label{Lem:Cal_explicit_representation}
In the given situation, let $X$ be a solution to \eqref{eq:SumFP} and denote by $F$ its distribution function, \textit{i.e.}, $F(t) = \Prob(X \leq t)$. Let $(W_1,W_2,\nu)$ be as in Proposition \ref{Prop:Disintegration}. Then, almost surely and for any continuity point of $\nu$,
\begin{align}
\nu((-\infty,t])\ &=\ \lim_{n \to \infty} \sum_{|v|=n} F(t/L(v)),&\text{if}\quad t < 0
\label{eq:nu-}	\\
\text{and}\quad
\nu([t,\infty))\ &=\ \lim_{n \to \infty} \sum_{|v|=n} (1\!-\! F(t/L(v))),&\text{if}\quad t > 0.
\label{eq:nu+}
\end{align}
Furthermore,
\begin{align}
\lim_{\varepsilon \to 0} & \liminf_{n \to \infty}
\sum_{|v|=n} L(v)^2 \left(\int_{\{|x|<\varepsilon/L(v)\}} \!\!\!\!\! x^2 \, F(\dx) - \, \left[\int_{\{|x|<\varepsilon/L(v)\}} \!\!\!\!\! x \, F(\dx) \right]^2 \right)\quad	\notag	\\
& =~
\lim_{\varepsilon \to 0} \limsup_{n \to \infty}
\sum_{|v|=n} L(v)^2 \left( \int_{\{|x|<\varepsilon/L(v)\}} \!\!\!\!\! x^2 \, F(\dx) - \! \left[\int_{\{|x|<\varepsilon/L(v)\}} \!\!\!\!\! x \, F(\dx) \right]^2 \right)	\notag	\\
& =~
W_2.	\label{eq:W_2}
\end{align}
Finally, if
\begin{equation}	\label{eq:W_1(tau)}
W_1(\tau)	~:=~	\lim_{n \to \infty} \sum_{|v|=n} L(v) \int_{\{|x|<\tau/L(v)\}} \!\!\!\!\! x \, F(\dx),
\end{equation}
for $\tau > 0$, then
\begin{equation}	\label{eq:W_1}
W_1	~=~	W_1(\tau) - \int_{\{|x|<\tau\}} \! \frac{x^3}{1+x^2} \, \nu(\dx)
+ \int_{\{|x| \geq \tau\}} \! \frac{x}{1+x^2} \, \nu(\dx)
\end{equation}
whenever $\tau$ and $-\tau$ are continuity points of $\nu$.
\end{Lemma}

While the measurability of $W_2$ and $\nu$ immediately follow from this lemma, the (random) points of discontinuity of $\nu$ are unknown at this point. Hence, a proof of the $\bL$-measurability of $W_1$ using the explicit representation of $W_1$ provided by \eqref{eq:W_1} could easily become messy. For this reason, this task is postponed until the end of Subsection \ref{subsec:Levy_measure} when we have derived the explicit form of $\nu$.
\end{proof}

The next result is a key to our further analysis and provides us with a functional equation for the disintegrated characteristic functions.

\begin{Lemma}[\textit{cf.}\ Lemma 5.2 in \cite{AM2009}]	\label{Lem:disint_FPE}
Let $\phi \in \Fsum$ and denote by $\Phi$ the disintegration of $\phi$. Then, for all $n \in \N_0$,
\begin{equation} \label{eq:disintegrated_FPE}
\Phi(t)	~=~	\prod_{|v|=n} [\Phi]_v(L(v)t)	\quad	\text{for all } t \in \R \text{ almost surely.}
\end{equation}
In particular, the characteristic exponent $\Psi$ of the corresponding disintegration $\Phi$ satisfies
\begin{equation}	\label{eq:Psi's_equation}
\Psi(t)	~=~	\sum_{|v|=n} [\Psi]_v(L(v)t)	\quad	\text{for all } t \in \R \text{ almost surely.}
\end{equation}
\end{Lemma}
\begin{proof}
For $t \geq 0$
\begin{eqnarray*}
\Phi(t)
& = &
\lim_{k \to \infty} \prod_{|v|=n+k} \phi(L(v) t) \\
& = &
\lim_{k \to \infty} \prod_{|v|=n} \prod_{|w|=k} \phi([L(w)]_v L(v) t) \\
& = &
\prod_{|v|=n} [\Phi]_v(L(v) t) \qquad	\text{almost surely}
\end{eqnarray*}
where we made use of the fact that $\mathcal{G}_n = \{v \in \N^n: L(v)>0\}$ is finite by assumption \eqref{eq:N<infty}. Since $\Phi$ and all the $[\Phi]_v$ are continuous functions in $t$ almost surely, this equation actually holds almost surely for all $t \geq 0$ simultaneously. As to \eqref{eq:Psi's_equation}, notice that by \eqref{eq:disintegrated_FPE}, $\exp(\Psi(t)) = \exp(\sum_{|v|=n} [\Psi]_v(L(v)t))$ for all $t \in \R$ almost surely, that is, $\Psi(t)$ and $\sum_{|v|=n} [\Psi]_v(L(v)t))$ are both continuous logarithms of $\Phi$. Since both functions assume the value $0$ at $0$, and since continuous logarithms of continuous curves in $\C \setminus \{0\}$ can differ only by constant multiples of $2 \pi \i$\!, \eqref{eq:Psi's_equation} must hold.
\end{proof}

\subsection{Disintegration along ladder lines}

As in \cite{ABM2010}, we use the concept of ladder lines when studying disintegrations. We are particularly interested in approximating a disintegration $\Phi$ not only via the corresponding multiplicative martingale $(\Phi_n)_{n \geq 0}$, but also by
\begin{equation}	\label{eq:Phi_T_u}
\Phi_{\mathcal{T}_u}(t)	~:=~	\prod_{v \in \mathcal{T}_u} \phi(L(v) t),	\quad	t \in \R,\ u \geq 0,
\end{equation}
where $\mathcal{T}_u$ is the first exit line of the interval $(-\infty,u]$, that is,
\begin{equation}
\mathcal{T}_u	~:=~	\{v \in \mathcal{G}: S(v)>u \text{ and } S(v|k) \leq u \text{ for } k=0,\ldots,|v|-1\}.
\end{equation}

\begin{Lemma}	\label{Lem:disintegration_via_HSL}
Given $\phi \in \Fsum$ with disintegration $\Phi$,
\begin{equation}	\label{eq:Convergence_of_Phi_T_u}
\lim_{u \to \infty} \Phi_{\mathcal{T}_u}(t)
~=~ \lim_{u \to \infty} \prod_{v \in \mathcal{T}_u} \phi(L(v)t)	~=~	\Phi(t)
\quad	\text{almost surely.}
\end{equation}
for any $t \in \R$, and outside a $\Prob$-null set the convergence holds for all $t$ simultaneously.
\end{Lemma}
\begin{proof}
That \eqref{eq:Convergence_of_Phi_T_u} holds along a fixed sequence $u_n \to \infty$ can be derived from the arguments in the proofs of Lemma 8.5 and Lemma 8.7(b) in \cite{ABM2010}. Further, since by assumption \eqref{eq:N<infty} and Lemma \ref{Lem:sup_L(v)_to_0} the points $S(v)$, $v \in \mathcal{G}$ do not accumulate on finite intervals in $\R$, $(\Phi_{\mathcal{T}_u}(t))_{u \geq 0}$ constitutes a right-continuous martingale and convergence holds outside a $\Prob$-null set for any sequence $u \to \infty$. Now let
\begin{equation*}
N^c ~:=~ \left\{\lim_{u \to \infty} \Phi_{\mathcal{T}_u}(t) = \Phi(t)	\text{ for all } t \in \Q\right\}.
\end{equation*}
Then $\Prob(N) = 0$.
Next, consider $\Phi_{\mathcal{T}_u}(\cdot)=\Phi_{\mathcal{T}_u}(\cdot,\bL)$ as a function of the family $\bL = (L(v))_{v \in \V}$. Note that $\mathcal{T}_u$ also depends on $\bL$, thus $\mathcal{T}_u=\mathcal{T}_u(\mathbf{L})$. Given a realisation $\mathbf{l} = (l(v))_{v \in \V}$ of $\bL$, $\Phi_{\mathcal{T}_u(\mathbf{l})}(\cdot,\mathbf{l})$ is the characteristic function of the sum $\sum_{v \in \mathcal{T}_u(\mathbf{l})} l(v) X(v)$, where the family $(X(v))_{v \in \V}$ is a family of i.i.d.\ copies of a random variable $X$ having characteristic function $\phi$. Since $\Prob(N^c) = 1$, $\Phi_{\mathcal{T}_u(\mathbf{l})}(t,\mathbf{l})$ converges to $\Phi(t,\mathbf{l})$ as $u \to \infty$ for all $t \in \Q$ and $\Prob(\bL \in \cdot)$-almost all $\mathbf{l}$. Now fix any $\mathbf{l}$ for which this convergence at the rationals hold. Further, fix an arbitrary sequence $(u_n)_{n \geq 0}$ such that $0 \leq u_n \to \infty$. Then every vague limit of a vaguely convergent subsequence of the distributions of $\sum_{v \in \mathcal{T}_{u_n}(\mathbf{l})} l(v) X(v)$, $n \geq 0$ has a characteristic function coinciding with $\Phi(t,\mathbf{l})$ at each rational point $t \not = 0$. Since characteristic functions of probability distributions are continuous on $\R$, the limit thus has characteristic function $\Phi(\cdot,\mathbf{l})$. By the direct half of L\'evy's continuity theorem, we then get $\Phi_{\mathcal{T}_{u_n}(\mathbf{l})}(t,\mathbf{l}) \to \Phi(t,\mathbf{l})$ as $n \to \infty$ for all $t \in \R$.
\end{proof}

Lemma \ref{Lem:disintegration_via_HSL} allows us to prove a useful extension of Lemma \ref{Lem:Cal_explicit_representation}.

\begin{Lemma}	\label{Lem:explicit_representation_along_ladder_lines}
Let $X$ be a solution to \eqref{eq:SumFP} with distribution function $F$, \textit{i.e.}, $F(t) = \Prob(X \leq t)$. Let further $(W_1,W_2,\nu)$ be the random L\'evy triple of the disintegration $\Phi$ of the characteristic function $\phi$ of $X$, see Proposition \ref{Prop:Disintegration}. Then, for any continuity point $t$ of $\nu$, almost surely,
\begin{align}
\nu((-\infty,t])\ &=\  \lim_{u \to \infty} \sum_{v \in \mathcal{T}_u} F(t/L(v)),	&\text{if}\quad t < 0		
\label{eq:nu-_ladder}	\\
\text{and}\quad\nu([t,\infty))\ &=\ \lim_{u \to \infty} \sum_{v \in \mathcal{T}_u} (1\!-\! F(t/L(v))),	&\text{if}\quad	t > 0.
\label{eq:nu+_ladder}
\end{align}
Finally, if
\begin{equation}	\label{eq:W_1(tau)_ladder}
W_1(\tau)	~:=~	\lim_{u \to \infty} \sum_{v \in \mathcal{T}_u} L(v) \int_{\{|x|<\tau/L(v)\}} \!\!\!\!\! x \, F(\dx),
\end{equation}
for $\tau>0$, then
\begin{equation}	\label{eq:W_1_ladder}
W_1	~=~	W_1(\tau) - \int_{\{|x|<\tau\}} \! \frac{x^3}{1+x^2} \, \nu(\dx)
+ \int_{\{|x| \geq \tau\}} \! \frac{x}{1+x^2} \, \nu(\dx)
\end{equation}
whenever $\tau$ and $-\tau$ are continuity points of $\nu$.
\end{Lemma}
\begin{proof}
For any fixed sequence $u_n \uparrow \infty$ and $\Prob(\bL \in \cdot)$-almost all $\mathbf{l} = (l(v))_{v \in \V} \in [0,\infty)^{\V}$, $((l(v)X(v))_{v \in \mathcal{T}_{u_n}(\mathbf{l})})_{n \geq 0}$ is an infinitesimal triangular array, see \cite[(2) on p.\,95]{GK1968}. We further infer from Lemma \ref{Lem:disintegration_via_HSL} that $\sum_{v \in \mathcal{T}_{u_n}(\mathbf{l})} l(v) X(v)$ converges weakly to the distribution with characteristic function $\Phi(\cdot,\mathbf{l})$ as $n \to \infty$. \eqref{eq:nu-_ladder}, \eqref{eq:nu+_ladder}, and \eqref{eq:W_1_ladder} can therefore be derived from \cite[Theorem 1 on p.\,116]{GK1968}, see \cite[(9) on p.\,84]{GK1968} concerning \eqref{eq:W_1_ladder}.
Finally note that according to Lemma \ref{Lem:disintegration_via_HSL}, the exceptional $\Prob(\bL \in \cdot)$-null set can be chosen independently of the particular sequence $(u_n)_{n \geq 0}$.
\end{proof}

\subsection{Endogenous fixed points}	\label{subsec:endogenous_FPs}

The concept of \emph{endogeny} is due to Aldous and Bandyopadhyay \cite[Definition 7]{AB2005}.
We paraphrase their definition slightly so that we can define \emph{endogenous fixed points} in the given context without introducing further notation.

Endogenous fixed points are special solutions to \eqref{eq:SumFP} with the property that all their randomness can be expressed in terms of the weights $L(v)$, $v \in \V$ with no further randomness needed.
Note that in comparison to \cite[Definition 8.2]{ABM2010} we allow endogenous fixed points to be both positive and negative with positive probability.

\begin{Def}	\label{Def:end_FP}
Let $\beta > 0$ and define $T^{\beta} := (T_j^{\beta})_{j \geq 1}$. A random variable $W_\beta$ (or its distribution) is called an \emph{endogenous fixed point of the smoothing transform with respect to (w.r.t.)\ $T^{\beta}$} if there exists a Borel measurable function $g: [0,\infty)^{\V} \to \R$ such that $W_\beta := g(\bL)$ and
\begin{equation}	\label{eq:end_FP}
W_\beta ~=~ \sum_{|v|=n} L(v)^{\beta} \, [W_\beta]_v	\quad	\text{almost surely}
\end{equation}
for all $n \geq 0$.
$W_\beta$ is called non-trivial if $\Prob(W_\beta\not=0) > 0$.
\end{Def}
Notice that by definition there is always the trivial endogenous fixed point $0$. From \cite[Theorem 6.2(a)]{ABM2010}, we infer that under \eqref{eq:A1}-\eqref{eq:A4} there exists a unique (up to a positive scaling constant) non-trivial non-negative endogenous fixed point w.r.t.\ $T^{\alpha}$. For ease of reference, we state the uniqueness result for non-negative endogenous fixed points from \cite{ABM2010}:

\begin{Prop}	\label{Prop:uniqueness_of_end_FP}
Suppose that \eqref{eq:A1}-\eqref{eq:A4} hold true. Let $W_{\alpha}$ be a non-negative endogenous fixed point w.r.t.\ $T^{\alpha}$. Then $W_{\alpha} = cW$ a.s.\ for some $c \geq 0$.
\end{Prop}
\begin{proof}[Source]
This is Theorem 6.2(a) in \cite{ABM2010}.
\end{proof}

For the rest of this article, we fix one particular non-trivial, non-negative endogenous fixed point w.r.t.\ $T^{\alpha}$ and denote it by $W$, which complies with and thus only further specifies our previous choice of $W$ being a fixed random variable solving \eqref{eq:FPE_W}.
Henceforth, let
\begin{equation}	\label{eq:varphi}
\varphi(t)	~:=~	\E e^{-tW},
\quad	t \geq 0
\end{equation}
denote the Laplace transform of $W$. From Theorem 3.1 in \cite{ABM2010} it follows that $1\!-\!\varphi(t)$ is regularly varying of index $1$ at the origin. Equivalently,
\begin{equation}	\label{eq:D_1}
D_1(t)	~:=~	\frac{1\!-\!\varphi(t)}{t},
\quad	t>0
\end{equation}
is slowly varying at the origin. $W$ can be explicitly constructed from $\varphi$ via
\begin{eqnarray}
W
& = &
\lim_{n \to \infty} \sum_{|v|=n} 1\!-\!\varphi(L(v)^{\alpha})
~=~	\lim_{t \to \infty}	\sum_{v \in \mathcal{T}_t} 1\!-\!\varphi(L(v)^{\alpha})	\label{eq:W<->varphi}	\\
& = &
\lim_{t \to \infty} D_1(e^{-\alpha t}) \sum_{v \in \mathcal{T}_t} L(v)^{\alpha}
~=~ \lim_{t \to \infty} D_1(e^{-\alpha t}) W_{\mathcal{T}_t}^{(\alpha)}					\label{eq:W<->varphi_SH}
\ \text{a.s.,}
\end{eqnarray}
where $W_{\mathcal{T}_t}^{(\alpha)} := \sum_{v \in \mathcal{T}_t} L(v)^{\alpha} = \sum_{v \in \mathcal{T}_t} e^{-\alpha S(v)}$. \eqref{eq:W<->varphi} follows from Theorem 6.2 and Theorem 8.3 and Lemma 8.7 in \cite{ABM2010}. \eqref{eq:W<->varphi_SH} follows from Theorem 10.2 in the same reference.  $(W_n^{(\alpha)})_{n \geq 0} := (\sum_{|v|=n} L(v)^{\alpha})_{n \geq 0}$ is a non-negative martingale sometimes called Biggins' martingale. Given \eqref{eq:A1}-\eqref{eq:A3} and \eqref{eq:A4a}, the distinguished endogenous fixed point $W$ equals a positive constant times the limit of Biggins' martingale since $\varphi$ possesses a finite derivative at $0$ in this case. In this situation, for convenience, we assume $W = W^{(\alpha)}$.

We will return to endogenous fixed points in Subsection \ref{subsec:endogeny_disintegration} and show there the important extension of the above proposition that under \eqref{eq:A1}-\eqref{eq:A5} endogenous fixed points exist only w.r.t.\ to $T^{\alpha}$ and that they are always non-negative or non-positive.

\subsection{Identifying the random L\'evy measure}	\label{subsec:Levy_measure}

\begin{Lemma}	\label{Lem:Disintegration_evaluated}
Suppose that \eqref{eq:A1}-\eqref{eq:A4} hold.
Let $ \phi \in \Fsum$ with disintegration $\Phi = \exp(\Psi)$, where
\begin{equation*}
\Psi(t)	~=~	\i W_1t - \frac{W_2 t^2}{2}
+ \int \left(e^{\i tx} - 1 - \frac{\i tx}{1+x^2} \right) \, \nu(\dx),
\quad	t \in \R
\end{equation*}
as in \eqref{eq:char_exponent} in Proposition \ref{Prop:Disintegration}. Then, for $t > 0$,
\begin{equation}	\label{eq:nu_identified}
\nu([t,\infty)) = W c_1 t^{-\alpha} \quad \text{and} \quad \nu((-\infty,-t])  = W c_2 t^{-\alpha}
\end{equation}
for the fixed non-negative endogenous fixed point $W$ w.r.t.\ $T^{\alpha}$ and constants
$c_1,c_2 \geq 0$.
Moreover, if $\alpha \geq 2$, then $\nu = 0$ almost surely.
\end{Lemma}
\begin{proof}
By \eqref{eq:char_exponent} and \eqref{eq:Psi's_equation},
\begin{align*}
\i W_1t & - \frac{W_2 t^2}{2}
+ \int \left(e^{\i tx} - 1 - \frac{\i tx}{1+x^2} \right) \, \nu(\dx)	\\
& =~
\i \sum_{|v|=n} L(v) [W_1]_v t ~-~ \frac{\sum_{|v|=n} L(v)^2 [W_2]_v t^2}{2}	\\
& \hphantom{=}~ + \sum_{|v|=n} \int \left(e^{\i L(v)t x} - 1 - \frac{\i L(v)t x}{1+x^2} \right) \, [\nu]_v(\dx)	\\
& =~
\i t\sum_{|v|=n} L(v)  \left([W_1]_v + \int \left[\frac{x}{1+(L(v)x)^2} - \frac{x}{1+x^2} \right] \, [\nu]_v(\dx) \right)	\\
&	\hphantom{=}~ - \frac{\sum_{|v|=n} L(v)^2 [W_2]_v t^2}{2}	\\
& \hphantom{=}~ + \sum_{|v|=n} \int \left(e^{\i L(v)t x} - 1 - \frac{\i L(v)t x}{1+(L(v)x)^2} \right) \, [\nu]_v(\dx),
\quad	t \in \R.
\end{align*}
From the uniqueness of the L\'evy triple we particularly infer that
\begin{align}
%	W_1 &	~=~\sum_{|v|=n} L(v) \left([W_1]_v + \int \left[\frac{x}{1+(L(v)x)^2} - \frac{x}{1+x^2} \right] \, [\nu]_v(\dx)\right),	\label{eq:W_1}	\\
W_2  ~&=~	\sum_{|v|=n} L(v)^2 [W_2]_v,	\label{eq:W_2_endogenous}	\\
\int g(x) \, \nu(\dx)	~&=~	\sum_{|v|=n} \int g(L(v) x) \, [\nu]_v(\dx)	\label{eq:nu}
\end{align}
almost surely for all $n \geq 0$ and all non-negative Borel-measurable functions $g$ on $\R$.
Now consider the function $g(x) = \1_{[t,\infty)}(x)$ for any fixed $t > 0$. Then \eqref{eq:nu} turns into
\begin{equation*}
\nu([t,\infty))	~=~	\sum_{|v|=n} [\nu]_v([L(v)^{-1}t,\infty))
\end{equation*}
almost surely for all $n \geq 0$. Defining $f:[0,\infty) \to [0,1]$ by
\begin{equation*}
f(t)	~=~	
		\begin{cases}
		1									&	\text{if } t = 0,	\\
		\E \exp(-\nu([t^{-1},\infty)))			&	\text{if } t > 0,
		\end{cases}
\end{equation*}
gives a decreasing function with $\lim_{t \to 0} f(t) = 1 = f(0)$. If $f(t) = 1$ for all $t \geq 0$, then $\nu$ assigns no mass to $(0,\infty)$ and we can choose $c_1 = 0$ in \eqref{eq:nu_identified}. If $f(t) < 1$ for some $t > 0$ or, equivalently, if $\nu$ assigns positive mass to the positive halfline, then $f$ is a monotone, non-trivial solution to the functional equation \eqref{eq:SumFE}. Let $M_n(t) = \prod_{|v|=n} f(L(v)t)$ denote the corresponding multiplicative martingale with almost sure limit $M(t)$, see \cite[Lemma 8.1]{ABM2010}. From \cite[Theorem 8.3]{ABM2010}, we infer that $M(t)$ has the form
\begin{equation*}
M(t)	~=~	\exp(-W c_1 t^{\alpha})	\quad	\text{almost surely}
\end{equation*}
for some $c_1>0$ and each $t \geq 0$.
Using that $\exp(-\nu([t^{-1},\infty)))$ is bounded and $\bL$-measurable and the fact that $f(t) = \E M(t)$,  see \cite[Lemma 8.1]{ABM2010}, we infer that
\begin{eqnarray*}
\exp(-\nu([t^{-1},\infty)))
& = &
\lim_{n \to \infty} \E[\exp(-\nu([t^{-1},\infty))) \, |\, \A_n]	\\
& = &
\lim_{n \to \infty} \E\Big[\prod_{|v|=n} \exp(-[\nu]_v([(L(v)^{-1}t^{-1},\infty))) \, \big|\, \A_n\Big]	\\
& = &
\lim_{n \to \infty} \prod_{|v|=n} f(L(v)t)	\\
& = &
M(t)	~=~	\exp(-W c_1 t^{\alpha})	\quad	\text{almost surely.}
\end{eqnarray*}
In particular, $\nu([t,\infty)) = W c_1 t^{-\alpha}$ almost surely for all $t > 0$. An analogous argument yields that $\nu((-\infty,t]) = W c_2 |t|^{-\alpha}$, for some $c_2 \geq 0$ and all $t<0$.
Since $\nu$ is a random L\'evy measure, it particularly almost surely satisfies \eqref{eq:Levy_measure} necessitating that $\alpha \in (0,2)$ or $\nu = 0$ almost surely.
\end{proof}

Now we can close the gap in the proof of Proposition \ref{Prop:Disintegration} and show that $W_1$ is $\bL$ measurable:
\begin{proof}[Completion of the proof of Proposition \ref{Prop:Disintegration}]
Since the random L\'evy measure is almost surely continuous w.r.t.\ to Lebesgue measure, we can choose an arbitrary $\tau > 0$ to calculate $W_1$ from \eqref{eq:W_1(tau)} and \eqref{eq:W_1}. It is easy to check that the right-hand side of \eqref{eq:W_1} is $\bL$-measurable.
\end{proof}

\subsection{Tail bounds for the fixed points}

In this subsection, we fix a solution $X$ to \eqref{eq:SumFP} with distribution function $F(t) = \Prob(X \leq t)$ and Fourier transform $\phi$. From Proposition \ref{Prop:Disintegration}, we know that $\phi$ has a disintegration of the form $\Phi=\exp(\Psi)$ with $\Psi$ being the (random) characteristic exponent of an infinitely divisible distribution with L\'evy triple $(W_1,W_2,\nu)$. From Lemma \ref{Lem:Disintegration_evaluated}, we infer that $\nu$ is of the form
\begin{equation}	\label{eq:nu[t,infty)}
\nu([t,\infty))	~=~	c_1 W t^{-\alpha}	\quad	\text{almost surely}	\quad (t > 0)
\end{equation}
where $c_1 \geq 0$, and $W$ is the unique non-trivial non-negative endogenous fixed point w.r.t.\ $T^{\alpha}$. 
Analogously,
\begin{equation}	\label{eq:nu(-infty,t]}
\nu((-\infty,t])	~=~	c_2 W |t|^{-\alpha}	\quad	\text{almost surely}	\quad (t < 0)
\end{equation}
where $c_2 \geq 0$.
Using these results as a starting point, we derive tail bounds for $X$ by comparing the tail probabilities of $X$ with the behaviour of the Laplace transform $\varphi$ of $W$ at $0$. In what follows, we are interested in
\begin{equation*}
\Ku	~:=~	\limsup_{t \to \infty} \frac{\Prob(|X|>t)}{1\!-\!\varphi(t^{-\alpha})}.
\end{equation*}

\begin{Lemma}	\label{Lem:tail_bounds}
Suppose that \eqref{eq:A1}-\eqref{eq:A4} hold. Then, in the given situation, the following assertions hold:
\begin{itemize}
	\item[(a)]
		$0 \leq \Ku < \infty$.
	\item[(b)]
		If $c_1+c_2 = 0$, then $\Ku = 0$.
\end{itemize}
\end{Lemma}
\begin{proof}
From Lemma \ref{Lem:explicit_representation_along_ladder_lines}, \eqref{eq:nu[t,infty)}, and \eqref{eq:nu(-infty,t]}, we infer that
\begin{eqnarray}
c_1 W t^{-\alpha}	& = &	\lim_{u \to \infty} \sum_{v \in \mathcal{T}_u} (1\!-\! F(t/L(v))),
\quad	\text{and}	\label{eq:(11.9)v1}	\\
c_2 W t^{-\alpha}	& =	&	\lim_{u \to \infty} \sum_{v \in \mathcal{T}_u} F(-t/L(v))	\label{eq:(11.9)v2}
\end{eqnarray}
almost surely for any $t>0$. Then, with $D(x) := x^{-\alpha} \Prob(|X|>x^{-1})$ and for $t=1$, we infer
\begin{equation}	\label{eq:(11.9)}
\lim_{u \to \infty} \sum_{v \in \mathcal{T}_u} e^{-\alpha S(v)} D(e^{-S(v)})	~=~	(c_1+c_2) W	\quad	\text{almost surely.}
\end{equation}
\eqref{eq:(11.9)} is the analogue to formula (11.6) in \cite{ABM2010}. Since, furthermore, $\Prob(|X|> x^{-1})$ is decreasing in $x$, Lemma 11.4 in \cite{ABM2010} also holds for $D$ (instead of $D_{\alpha}$ there). As pointed out right before Lemma 11.4 in \cite{ABM2010}, these two properties, namely, that $D$ satisfies \eqref{eq:(11.9)} and the assertion of Lemma 11.4 in \cite{ABM2010}, are the only properties needed in the proof of Lemma 11.5 in \cite{ABM2010}. Therefore, arguing as in the proof of Lemma 11.5 in \cite{ABM2010}, we can conclude the analogue of (11.9) there, namely,
\begin{equation*}
(c_1+c_2) W ~\geq~ e^{-\delta} \Ku (1-\varepsilon) W
\end{equation*}
almost surely for some $\delta > 0$ and $\varepsilon \in (0,1)$. 
From this, assertions (a) and (b) immediately follow since $W$ is almost surely finite and positive with positive probability.
\end{proof}

\subsection{Asymptotic results for general branching processes}	\label{subsec:HS_norming}

As in \cite[Section 9]{ABM2010}, we make use of results concerning the asymptotic behaviour of general branching processes derived from the BRW $(\mathcal{Z}_n)_{n \geq 0}$. 

Our first result in this section is on ratio convergence on $\Surv$ of certain general branching processes.

\begin{Prop}	\label{Prop:ratio_convergence}
Assume that \eqref{eq:A1}-\eqref{eq:A3} and \eqref{eq:A4b} hold and let $\varepsilon > 0$. Then for any $\beta > \theta$ and all sufficiently large $c$ (which may depend on $\beta$),
\begin{equation}	\label{eq:ratio_convergence}
\frac{\sum_{v \in \mathcal{T}_t} e^{-\beta (S(v)-t)}(S(v)-t))\1_{\{S(v) > t+c\}}}
{\sum_{v \in \mathcal{T}_t} e^{-\alpha (S(v)-t)} (S(v)-t)}
~\to~	\varepsilon(c) \leq \varepsilon
\text{ on } \Surv
\end{equation}
almost surely as $t \to \infty$.
\end{Prop}
\begin{proof}
Since the sum in the numerator is decreasing in $\beta$, we can w.l.o.g.\ assume that $\theta < \beta < \alpha$. Further, notice that the sums over $v \in \mathcal{T}_t$ in \eqref{eq:ratio_convergence} remain unaffected when replacing the underlying BRW $(\mathcal{Z}_n)_{n \geq 0}$ by the embedded BRW $(\mathcal{Z}^>_n)_{n \geq 0}$ the construction of which has been described in Subsection \ref{subsec:embedded}. This is due to the fact that the first crossing of the level $t$ necessarily takes place at a vertex $v$ such that $S(v)$ is a strict record in the finite sequence $0, S(v|1), \ldots, S(v)$. Therefore and in view of Proposition \ref{Prop:embedded_BRW}, it constitutes no loss of generality to assume that $\Prob(\mathcal{Z}((-\infty,0))>0) = 0$ holds true beyond \eqref{eq:A1}-\eqref{eq:A3} and \eqref{eq:A4b}.
We show that all assumptions of Theorem 6.3 in \cite{Ner1981} are fulfilled: \eqref{eq:A1} ensures that $\mathcal{Z}$ is non-lattice, \eqref{eq:A2} the supercriticality. The existence of a Malthusian parameter follows from \eqref{eq:A3}, while \eqref{eq:A4b} and the fact that $\mathcal{Z}$ is concentrated on the positive halfline imply that $m'(\alpha) \in (-\infty,0)$, which is Nerman's assumption (1.5). Moreover, \eqref{eq:A4b} implies Nerman's condition 6.1.
What remains to show is that numerator and denominator in \eqref{eq:ratio_convergence} derive from characteristics that satisfy Condition 6.2 in \cite{Ner1981}. To this end, note that, following Nerman's notation, the numerator is derived from
\begin{equation*}
\phi(t)	~=~	\1_{[0,\infty)}(t) \sum_{j=1}^N e^{-\beta(S(j)-t)} (S(j)-t) \1_{\{S(j) > t+c\}},
\end{equation*}
while the denominator is derived from
\begin{equation*}
\psi(t)	~=~	\1_{[0,\infty)}(t) \sum_{j=1}^N e^{-\alpha(S(j)-t)} (S(j)-t) \1_{\{S(j) > t\}}.
\end{equation*}
Plainly, $\phi$ and $\psi$ have c\`adl\`ag paths and $\E \phi(t)$ and $\E \psi(t)$ are continuous almost everywhere w.r.t.\ Lebesgue measure. Furthermore,
\begin{equation*}
e^{-\beta t} \phi(t)
~=~	\1_{[0,\infty)}(t) \sum_{j=1}^N e^{-\beta S(j)} (S(j)-t) \1_{\{S(j) > t+c\}}
~\leq~	\sum_{j=1}^N e^{-\beta S(j)} S(j)
\end{equation*}
which is integrable by assumption \eqref{eq:A4b}, for $\beta > \theta$. Thus, $\phi$ satisfies Nerman's Condition 6.2. Analogously, one can deduce that $\psi$ satisfies the same condition. Now applying Theorem 6.3 in \cite{Ner1981}, we infer that the ratio in \eqref{eq:ratio_convergence} tends to
\begin{equation*}
\frac{\int_0^{\infty} e^{-\alpha x} \E \phi(x) \, \dx}{\int_0^{\infty} e^{-\alpha x} \E \psi(x) \, \dx}
~\leq~
\frac{\int_0^{\infty} \E \sum_{j=1}^N e^{-\beta S(j)} (S(j)-x) \1_{\{S(j) > x+c\}} \, \dx}{\int_0^{\infty} \E \sum_{j=1}^N e^{-\alpha S(j)} (S(j)-x) \1_{\{S(j) > x\}} \, \dx}
~=:~	\varepsilon(c)
\end{equation*}
almost surely on $\Surv$ as $t \to \infty$. (Notice that we have used that $\beta < \alpha$ to derive the inequality for the numerator.) This completes the proof since $\varepsilon(c) \to 0$ as $c \to \infty$.
\end{proof}

Later, we will need the following result that may be viewed a kind of converse of Theorem 9.4 in \cite{ABM2010}:

\begin{Thm}	\label{Thm:reverse_version}
Suppose that \eqref{eq:A1}-\eqref{eq:A3} and \eqref{eq:A4b} hold. Assume also that the following conditions hold:
\begin{itemize}
	\item[(i)]
		There are a non-negative function $H$ and a random variable $\widetilde{W}$ such that
		\begin{equation*}
		H(t) \sum_{v \in \mathcal{T}_t} e^{-\alpha S(v)} (S(v)-t)	~\to~	\widetilde{W}
		\end{equation*}
		almost surely as $t \to \infty$.
	\item[(ii)]
		For some $h \in (0,\infty)$
		\begin{equation*}
		\varepsilon_t(a)	~=~	\frac{H(a+t)}{H(t)}-h	~\to~	0
		\end{equation*}
		as $t \to \infty$ uniformly on compact subsets of $[0,\infty)$.
	\item[(iii)]
		For a finite $K$, some $\theta < \beta < \alpha$, all $a \geq 0$, and all sufficiently large $t>0$,
		\begin{equation*}
		\frac{H(a+t)}{H(t)}	~\leq~	K e^{(\alpha-\beta)a}.
		\end{equation*}
\end{itemize}
Then
\begin{equation}	\label{eq:reverse_version}
\sum_{v \in \mathcal{T}_t} e^{-\alpha S(v)} H(S(v)) (S(v)-t)	~\to~	h^{-1}\widetilde{W}
\end{equation}
almost surely as $t \to \infty$.
\end{Thm}
\begin{proof}
A proof similar to the proof of Theorem 9.4 in \cite{ABM2010} gives the result. We therefore refrain from giving details here.
%In this proof, $\sum$ denotes the sum over $v \in \mathcal{T}_t$. It is clear that the assertion holds on $\Surv^c$ so that we can focus on what happens on the survival set $\Surv$ in what follows. By increasing $K$ if necessary, we can assume that for all sufficiently large $t$ we have
%\begin{equation*}
%\varepsilon_t(a)	~\leq~	K e^{(\alpha-\beta)a}
%\end{equation*}
%for all $a \geq 0$. Similar as in the proof of Theorem 9.3 in \cite{ABM2010}, we consider the ratio
%\begin{eqnarray*}
%\frac{\sum e^{-\alpha S(v)} H(S(v)) (S(v)-t)}{H(t) \sum e^{-\alpha S(v)} (S(v)-t)}
%& = &
%\frac{\sum e^{-\alpha S(v)} H(S(v))/H(t) (S(v)-t)}{\sum e^{-\alpha S(v)} (S(v)-t)}	\\
%& = &
%\frac{\sum e^{-\alpha S(v)} (h+\varepsilon_t(S(v)-t)) (S(v)-t)}{\sum e^{-\alpha S(v)} (S(v)-t)}	\\
%& = &
%h + \frac{\sum e^{-\alpha S(v)} \varepsilon_t(S(v)-t) (S(v)-t)}{\sum e^{-\alpha S(v)} (S(v)-t)}.
%\end{eqnarray*}
%The fraction on the right-hand side can be estimated as follows:
%\begin{align*}
%\Bigg| & \frac{\sum e^{-\alpha S(v)} \varepsilon_t(S(v)-t) (S(v)-t)}{\sum e^{-\alpha S(v)} (S(v)-t)} \Bigg|	\\
%& \quad \leq ~
%\sup_{0 \leq a \leq c} |\varepsilon_t(a)| +
%\frac{\sum e^{-\alpha S(v)} \varepsilon_t(S(v)-t) (S(v)-t) \1_{\{S(v)>t+c\}}}{\sum e^{-\alpha S(v)} (S(v)-t)}.
%\end{align*}
%Here, the first term on the right-hand side tends to $0$ by Condition (ii) whereas the second term converges to $0$ as $t$ and then $c$ tend to infinity by Proposition \ref{Prop:ratio_convergence}.
\end{proof}

\subsection{Endogeny and disintegration}	\label{subsec:endogeny_disintegration}

The following two theorems are on endogenous fixed points and they constitute the main results of this subsection.

\begin{Thm}	\label{Thm:no_end_T^beta}
Assume that \eqref{eq:A1}-\eqref{eq:A4} hold and that $\widetilde{W}$ is an endogenous fixed point w.r.t.\ $T^{\beta}$ for some $\beta > 0$, $\beta \not= \alpha$. Then $\widetilde{W} = 0$ almost surely.
\end{Thm}

For the second theorem, recall that $W$ denotes a fixed non-trivial non-negative endogenous fixed point w.r.t.\ $T^{\alpha}$.

\begin{Thm}	\label{Thm:endogeneuous=>one_sided}
If \eqref{eq:A1}-\eqref{eq:A5} hold, then any endogenous fixed point $W_{\alpha}$ w.r.t.\ $T^{\alpha}$ satisfies $W_{\alpha} = cW$ almost surely for some $c\in\R$.
\end{Thm}

Some preliminary work is needed to prove these results.

\begin{Lemma}	\label{Lem:W_n^1}
Assume that \eqref{eq:A1}-\eqref{eq:A4} hold.
\begin{itemize}
	\item[(a)]
		If $\alpha < 1$, then $\sum_{|v|=n} L(v) \to 0$ almost surely.
	\item[(b)]
		If $\alpha > 1$, then, almost surely as $n \to \infty$,
		\begin{equation}	\label{eq:W_n^1}
		\sum_{|v|=n} L(v)	~\to~
		\begin{cases}
		\infty	&	\text{on }	\Surv	\\
		0				&	\text{on }	\Surv^c.
		\end{cases}
		\end{equation}
\end{itemize}
\end{Lemma}
\begin{proof}
(a)
If $\alpha < 1$, then
\begin{equation*}
\sum_{|v|=n} L(v)	~\leq~	W_n^{(\alpha)} \sup_{|v|=n} L(v)^{1-\alpha}	~\to~	0
\end{equation*}
almost surely because $(W_n^{(\alpha)})_{n \geq 0} = (\sum_{|v|=n} L(v)^{\alpha})_{n \geq 0}$ is a non-negative martingale and $\sup_{|v|=n} L(v)$ $ \to 0$ by Lemma \ref{Lem:sup_L(v)_to_0}.

\noindent
(b)
Clearly, $\sum_{|v|=n} L(v) \to 0$ as $n \to \infty$ on $\Surv^c$ whence it remains to prove that this sum tends to $\infty$ almost surely on $\Surv$. To this end, recall from \eqref{eq:W<->varphi_SH} that
\begin{equation}	\label{eq:sum_to_W}
\sum_{|v|=n} 1\!-\!\varphi(L(v)^{\alpha})	~=~	\sum_{|v|=n} L(v)^{\alpha} D_1(L(v)^{\alpha})	~\to~	W
\end{equation}
almost surely as $n \to \infty$, where $\varphi$ denotes the Laplace transform of $W$ and $D_1(t) = t^{-1}(1\!-\!\varphi(t))$, $t > 0$. $D_1(t)$ is slowly varying as $t \to 0$. Now fix any $\delta > 0$ such that $1+\delta < \alpha$. By Potter's Theorem \cite[Theorem 1.5.6]{BGT1989}, we infer that $D_1(t^{\alpha}) \leq 2t^{-\delta}$ for all sufficiently small $t$. Thus, using Lemma \ref{Lem:sup_L(v)_to_0}, we get
\begin{eqnarray*}
\sum_{|v|=n} L(v)
& = &
\sum_{|v|=n} L(v)^{\alpha} L(v)^{-\delta} L(v)^{1+\delta-\alpha}	\\
& \geq &
\frac{1}{2} \left(\sup_{|v|=n} L(v) \right)^{1+\delta-\alpha} \left(\sum_{|v|=n} L(v)^{\alpha} D_1(L(v)^{\alpha}) \right)
\end{eqnarray*}
for sufficiently large $n$. Now $\sup_{|v|=n} L(v)$ tends to $0$ almost surely as $n \to \infty$ by Lemma \ref{Lem:sup_L(v)_to_0}, while the last factor tends to $W$ almost surely by \eqref{eq:sum_to_W}. Therefore, the desired conclusion follows from the fact that $W > 0$ almost surely on $\Surv$, which is a commonplace in the study of \eqref{eq:SumFP} on the positive halfline (and is an immediate consequence of the fact that $\Prob(W=0)$ is a fixed point of the function $f(s) = \E s^{N}$ in $[0,1)$).
\end{proof}

\begin{Lemma}	\label{Lem:end_1<alpha}
Assume that $\alpha > 1$ and that $W_1$ is an endogenous fixed point w.r.t.\ $T$ with finite mean. Then $W_1 = 0$ almost surely.
\end{Lemma}
\begin{proof}
Since $W_1$ is endogenous, it is in particular $\A_{\infty}$-measurable. Thus, by the integrability of $W_1$, we have
\begin{align}	\label{eq:integrable_W_1}
W_1	~&=~ \lim_{n \to \infty} \E[W_1 | \A_n]\nonumber\\
&=~ \lim_{n \to \infty} \E\Bigg[ \sum_{|v|=n} L(v) [W_1]_v \,\Bigg|\, \A_n \Bigg]
~=~	(\E W_1) \sum_{|v|=n} L(v)
\end{align}
almost surely. By Lemma \ref{Lem:W_n^1}, $\sum_{|v|=n} L(v) \to \infty$ almost surely on $\Surv$ whereas $|W_1|<\infty$ almost surely. Therefore, $\E W_1 = 0$ must hold which, in combination with \eqref{eq:integrable_W_1}, implies $W_1 = 0$ almost surely.
\end{proof}

\begin{Lemma}	\label{Lem:linear_combination_endogeny}
Suppose that \eqref{eq:A1}-\eqref{eq:A4} hold and that $W_{\alpha}$ is an endogenous fixed point w.r.t.\ $T^{\alpha}$. If $a W_{\alpha} + b W \geq 0$ almost surely for some constants $0 \not = a \in \R$, $b \in \R$, then $W_{\alpha} = c W$ almost surely for some $c \in \R$.
%\begin{itemize}
%	\item[(a)]
%		If $a W_{\alpha} + b W \geq 0$ almost surely for some constants $0 \not = a \in \R$,
%		$b \in \R$, then $W_{\alpha} = c W$ almost surely for some $c \in \R$.
%	\item[(b)]
%		If $|W_{\alpha}| \leq C W$ almost surely for some $C < \infty$, then
%		$W_{\alpha} = c W$ almost surely for some $c \in \R$.
%\end{itemize}
\end{Lemma}
\begin{proof}
If $a W_{\alpha} + b W \geq 0$, then $a W_{\alpha} + b W$ is a non-negative endogenous fixed point and, therefore,  by Proposition \ref{Prop:uniqueness_of_end_FP}, equals $cW$ for some $c \geq 0$, from which the desired conclusion follows.
%\noindent
%(b) If $|W_{\alpha}| \leq C W$ almost surely for some $C < \infty$, then $W_{\alpha} + C W \geq 0$ almost surely and assertion (a) yields $W_{\alpha} = c W$ almost surely for some $c \in \R$.
\end{proof}

Our next result connects the concepts of disintegration and endogeny.
In what follows, we call a disintegration $\Phi$ \emph{almost surely degenerate} iff there exists an $\bL$-measurable random variable $W_1$ such that $\Phi(t) = \exp(\i W_1 t)$ for all $t \in \R$ almost surely.

\begin{Prop}	\label{Prop:Endogeny_disintegration}
Let $P \in \Fsum$ with disintegration $\Phi$. Then there exists an endogenous fixed point $W_1$ w.r.t.\ $T$ with distribution $P$ iff $\Phi$ is almost surely degenerate. Further, in this case, $\Phi(t) = \exp(\i W_1 t)$ almost surely for all $t \in \R$.
\end{Prop}
\begin{proof}
First assume that there exists $W_1 \stackrel{\mathrm{d}}{=} P$ which is endogenous w.r.t.\ $T$, thus
\begin{equation*}
W_1 ~=~	\sum_{|v|=n} L(v) [W_1]_v	\quad	\text{almost surely}
\end{equation*}
for all $n \geq 0$. Let $\phi$ be its characteristic function. Then, with $(\Phi_n)_{n \geq 0}$ denoting the corresponding multiplicative martingale, we have
\begin{eqnarray*}
\E[\exp(\i W_1 t) \,|\,\A_n]
& = &
\E\left[\exp\bigg(\i t\sum_{|v|=n} L(v) [W_1]_v \bigg) \,\bigg|\,\A_n\right]	\\
& = &
\E \left[ \prod_{|v|=n} \exp(\i [W_1]_v L(v)t) \,\bigg|\,\A_n\right]	\\
& = &
\prod_{|v|=n} \phi(L(v)t)
~=~	\Phi_n(t)	~\underset{n \to \infty}{\to}~	\Phi(t)	\quad	\text{almost surely.}
\end{eqnarray*}
On the other hand, by the boundedness of $\exp(\i W_1 t)$ and the martingale convergence theorem,
\begin{equation*}
\E[\exp(\i W_1 t) \,|\,\A_n]	~\to~	\exp(\i W_1 t)	\quad	\text{almost surely as } n \to \infty.
\end{equation*}
This proves that $\Phi(t) = \exp(\i W_1 t)$ almost surely and thus that the disintegration is the characteristic function of a Dirac measure almost surely.

Conversely, if $\phi \in \Fsum$ has a disintegration $\Phi$ of the form $\Phi(t) = \exp(\i W_1 t)$ for all $t \in \R$ almost surely for some $\bL$-measurable random variable $W_1$, then, by \eqref{eq:disintegrated_FPE}, for all $t \in \R$,
\begin{eqnarray*}
e^{\i W_1 t} 
& = &
\Phi(t)	~=~	\prod_{|v|=n} [\Phi]_v(L(v) t)
~=~	\prod_{|v|=n} e^{\i [W_1]_v L(v) t}	\\
& = &
\exp\bigg(\i \sum_{|v|=n} L(v) [W_1]_v t\bigg)	\quad	\text{almost surely.}
\end{eqnarray*}
Hence, by the uniqueness theorem for characteristic functions,
\begin{equation*}
W_1	~=~	\sum_{|v|=n} L(v) [W_1]_v	\quad	\text{almost surely}
\end{equation*}
which shows the endogeny of $W_1$.
\end{proof}

Now we are ready to prove Theorem \ref{Thm:no_end_T^beta} and Theorem \ref{Thm:endogeneuous=>one_sided}.

\begin{proof}[Proof of Theorem \ref{Thm:no_end_T^beta}]
Let $\beta > 0$, $\beta \not = \alpha$. Since \eqref{eq:A1}-\eqref{eq:A4} carry over from the sequence $T$ to the sequence $T^{\beta} = (T_j^{\beta})_{j \geq 1}$ (with the new $\alpha$ being $\alpha/\beta$), we can w.l.o.g.\ assume that $\beta = 1$ and $\alpha \not = 1$. Let $\widetilde{W}$ be an endogenous fixed point w.r.t.\ $T$. We must verify that $\widetilde{W} = 0$ almost surely.
By Proposition \ref{Prop:Endogeny_disintegration}, we infer that the disintegration $\Phi$ of $\widetilde{W}$ equals $\exp(\i \widetilde{W}t)$. In particular, the random L\'evy triple of $\Phi$ equals $(\widetilde{W},0,0)$. Thus, using \eqref{eq:W_1(tau)} and \eqref{eq:W_1}, we infer that
\begin{equation}
\widetilde{W}%	~=~	W_1	~=~	W_1(\tau)
~=~	\lim_{n \to \infty} \sum_{|v|=n} L(v) \int_{\{|x|<\tau/L(v)\}} \!\!\!\!\! x \, F(\dx)
\end{equation}
for arbitrary $\tau > 0$, where $F$ denotes the distribution function of $\widetilde{W}$. Recall from Lemma \ref{Lem:tail_bounds} that
$\limsup_{t \to \infty} \Prob(|\widetilde{W}|>t)/(1\!-\!\varphi(t^{-\alpha})) = 0$ with $\varphi$ denoting the Laplace transform of $W$, the fixed non-trivial non-negative endogenous fixed point. Integration by parts further yields
\begin{equation}	\label{eq:integration_by_parts}
\int_{\{|x|<t\}} \!\!\! |x| \, F(\dx)
~=~	\int_0^t \! \Prob(|\widetilde{W}|>x)\dx - t\Prob(|\widetilde{W}|>t).
\end{equation}
Now suppose first that $\alpha < 1$ and recall that $1\!-\!\varphi(t)$ is regularly varying of index $1$ at the origin. Choose an arbitrary $\varepsilon > 0$ and then $t > 0$ large enough such that $\Prob(|\widetilde{W}|>x) \leq \varepsilon (1\!-\!\varphi(x^{-\alpha}))$ for all $x \geq t$. Then, putting things together, we obtain
\begin{eqnarray*}
|\widetilde{W}|
& \leq &
\limsup_{n \to \infty} \sum_{|v|=n} L(v) \int_{\{|x|<\tau/L(v)\}} \!\!\!\!\! |x| \, F(\dx)	\\
& \leq &
\limsup_{n \to \infty} \sum_{|v|=n} L(v) \int_0^{\tau/L(v)} \!\!\! \Prob(|\widetilde{W}|>x) \, \dx	\\
& \leq &
\limsup_{n \to \infty} \sum_{|v|=n} L(v) \left( t+ \varepsilon \int_t^{\tau/L(v)} \big(1\!-\!\varphi(x^{-\alpha})\big) \, \dx \right).
\end{eqnarray*}
Here, $\sum_{|v|=n} L(v) \to 0$ almost surely as $n \to \infty$ by Lemma \ref{Lem:W_n^1}(a).
Hence, using Proposition 1.5.8 in \cite{BGT1989} and the fact that $1\!-\!\varphi(x^{-\alpha})$ is regularly varying of index $-\alpha$ at $\infty$ by \cite[Theorem 3.1]{ABM2010}, we arrive at
\begin{eqnarray*}
|\widetilde{W}|
& \leq &
\varepsilon \limsup_{n \to \infty} \sum_{|v|=n} L(v) \int_t^{\tau/L(v)} \!\!\! 1\!-\!\varphi(x^{-\alpha}) \, \dx	\\
& = &
\varepsilon \limsup_{n \to \infty} \sum_{|v|=n} L(v) \frac{\tau/L(v)}{1-\alpha} (1\!-\!\varphi((\tau/L(v))^{-\alpha}))	\\
& = &
\frac{\varepsilon \tau^{1-\alpha}}{1-\alpha} \limsup_{n \to \infty} \sum_{|v|=n} \big(1\!-\!\varphi(L(v)^{\alpha})\big)	\\
& = &
\frac{\varepsilon \tau^{1-\alpha}}{1-\alpha} \, W	\quad	\text{almost surely,}
\end{eqnarray*}
where the last equality follows from \eqref{eq:W<->varphi}. Letting $\varepsilon \to 0$ yields $|\widetilde{W}|=0$ almost surely.

If $\alpha > 1$, then Lemma \ref{Lem:tail_bounds} provides us with $\Prob(|\widetilde{W}| > t) = o(1\!-\!\varphi(t^{-\alpha}))$ as $t \to \infty$. Since $1\!-\!\varphi(t^{-\alpha})$ is regularly varying of index $-\alpha$ at infinity, we infer that $\E |\widetilde{W}|<\infty$ and thus $\widetilde{W} = 0$ almost surely by Lemma \ref{Lem:end_1<alpha}.
\end{proof}

\begin{proof}[Proof of Theorem \ref{Thm:endogeneuous=>one_sided}]
Without loss of generality let $\alpha = 1$. Then suppose that $W_1$ is an endogenous fixed point w.r.t.\ $T$. By Proposition \ref{Prop:Endogeny_disintegration}, its disintegration $\Phi$ is of the form $\Phi(t) = \exp(\i W_1 t)$ ($t \in \R$) almost surely. In particular, the random L\'evy triple corresponding to $\Phi$ equals $(W_1,0,0)$. Therefore, we infer from \eqref{eq:W_1(tau)_ladder} and \eqref{eq:W_1_ladder} that
\begin{eqnarray}
W_1
& = &
%W_1(1)
\lim_{t \to \infty} \sum_{v \in \mathcal{T}_t} L(v) \int_{\{|x|<L(v)^{-1}\}} \!\!\!\!\! x \, F(\dx)	\notag	\\
& = &
\lim_{t \to \infty} \sum_{v \in \mathcal{T}_t} L(v) I(L(v)^{-1})
\quad	\text{almost surely}	\label{eq:W_1=W_1(1)}
\end{eqnarray}
where $I(c) := \int_{\{|x|<c\}} x \, F(\dx)$ ($c\geq0$) and $F$ denotes the distribution function of $W_1$. Now suppose that
\begin{equation*}
K ~:=~ \limsup_{t \to \infty} I(t)/D_1(t^{-1}) ~<~ \infty.
\end{equation*}
Then
\begin{equation*}
W_1	~\leq~	K \lim_{t \to \infty} \sum_{v \in \mathcal{T}_t} L(v) D_1(L(v))	~=~	K W	
\end{equation*}
almost surely by \eqref{eq:W<->varphi}. Lemma \ref{Lem:linear_combination_endogeny}(b) then implies that $W_1 = aW$ for some $a \in \R$. Using an analogous argument, we arrive at the same conclusion if
\begin{equation*}
\liminf_{t \to \infty} I(t)/D_1(t^{-1}) ~>~ -\infty.
\end{equation*}
Therefore, it remains to consider the case
\begin{equation}	\label{eq:oscillating}
-\infty	~=~	\liminf_{t \to \infty} I(t)/D_1(t^{-1})
~<~	\limsup_{t \to \infty} I(t)/D_1(t^{-1})	~=~	\infty
\end{equation}
in which there are arbitrarily large $t$ such that $I(e^{t}) \leq 0$. For any such $t$, we have
\begin{eqnarray*}
\sum_{v \in \mathcal{T}_t} L(v) I(L(v)^{-1})
& \leq &
\sum_{v \in \mathcal{T}_t} L(v) (I(L(v)^{-1})-I(e^{t}))	\\
& \leq &
\sum_{v \in \mathcal{T}_t} L(v) \int_{\{e^t \leq |x| < L(v)^{-1}\}} |x| \, F(\dx)	\\
& = &
\sum_{v \in \mathcal{T}_t} L(v) \Bigg[ \int_{e^t}^{L(v)^{-1}} \Prob(|W_1|>x) \, \dx	\\
& &
- L(v)^{-1} \Prob(|W_1|>L(v)^{-1}) + e^{t} \Prob(|W_1|>e^{t})\Bigg]	\\
& \leq &
\sum_{v \in \mathcal{T}_t} L(v) \int_{e^t}^{L(v)^{-1}} \Prob(|W_1|>x) \, \dx	\\
& & + e^{t} \Prob(|W_1|>e^{t}) \sum_{v \in \mathcal{T}_t} L(v).
\end{eqnarray*}
%where we have used integration by parts to arrive at the penultimate line.
Since $W_1$ is endogenous, and therefore, $\nu = 0$ almost surely by Proposition \ref{Prop:Endogeny_disintegration} (with $\nu$ denoting the random L\'evy measure of the disintegration $\Phi$ of $X$), Lemma \ref{Lem:tail_bounds}(b) implies that $\Prob(|W_1|>x) = o(1\!-\!\varphi(x^{-1}))$ as $x \to \infty$ (with $\varphi$ denoting the Laplace transform of the distinguished non-negative endogenous fixed point $W$). Therefore,
\begin{equation*}
e^{t} \Prob(|W_1|>e^{t}) \sum_{v \in \mathcal{T}_t} L(v)
~=~ o(D_1(e^{-t})) \sum_{v \in \mathcal{T}_t} L(v)
~=~	o(W)
\end{equation*}
almost surely as $t \to \infty$ by \eqref{eq:W<->varphi_SH}. Consequently, for any $\varepsilon > 0$, we have that 
%(with the limit $t \to \infty$ subject to the constraint $I(e^{t}) \leq 0$)
\begin{eqnarray*}
W_1
& = &
\lim_{t \to \infty,\,I(e^{t}) \leq 0}\, \sum_{v \in \mathcal{T}_t} L(v) I(L(v)^{-1})	\\
& \leq &
\limsup_{t \to \infty}\, \sum_{v \in \mathcal{T}_t} L(v) \int_{e^t}^{L(v)^{-1}} \Prob(|W_1|>x) \, \dx	\\
& \leq &
\varepsilon \limsup_{t \to \infty}\, \sum_{v \in \mathcal{T}_t} L(v) \int_{e^t}^{L(v)^{-1}} x^{-1} D_1(x^{-1}) \, \dx.
\end{eqnarray*}
Now we have to distinguish two cases.

Suppose first that \eqref{eq:A4a} and \eqref{eq:A5} hold. Then $D_1(x^{-1})$ increases to $\E W$ as $x \to \infty$. By convention, $\E W = 1$. Therefore,
\begin{eqnarray}
W_1
& \leq &
\varepsilon \limsup_{t \to \infty}\, \sum_{v \in \mathcal{T}_t} L(v) \int_{e^t}^{L(v)^{-1}} x^{-1} \, \dx	\notag	\\
& \leq &
\varepsilon \limsup_{t \to \infty}\, \sum_{v \in \mathcal{T}_t} e^{-S(v)} (S(v)-t)	\notag	\\
& = &
\varepsilon \limsup_{t \to \infty} e^{-t} \sum_{v \in \mathcal{T}_t} e^{t-S(v)} (S(v)-t)	\notag	\\
& = &
\varepsilon \limsup_{t \to \infty} e^{-t} \sum_{v \in \V} [\phi]_v(t-S(v)),
\label{eq:Nerman}
\end{eqnarray}
where
\begin{equation}	\label{eq:phi}
\phi(t)	~:=~	\sum_{j \geq 1} e^{t-S(j)} \1_{[0,S(j))}(t)(S(j)-t).
\end{equation}
We further have that
\begin{eqnarray*}
\int_0^{\infty} e^{-t} \E \phi(t) \, \dt
& = &
\int_0^{\infty} \E \sum_{j \geq 1} e^{-S(j)} \1_{[0,S(j))}(t) (S(j)-t) \, \dt	\\
& = &
\E \sum_{j \geq 1} e^{-S(j)} S(j)^2/2 \, \dt
~<~	\infty
\end{eqnarray*}
by \eqref{eq:A5}.
Therefore, \cite[Theorem 3.1]{Ner1981} yields that the random series in \eqref{eq:Nerman} tends to $cW$ in probability as $t \to \infty$ for a suitable constant $c\in [0,\infty)$. Choosing an appropriate subsequence, we can assume that the convergence holds almost surely. Consequently, we have $W_1 \leq \varepsilon c W$. Letting $\varepsilon \to 0$, we obtain that $W_1 \leq 0$ and thus $I \leq 0$ which obviously contradicts \eqref{eq:oscillating}.

Now suppose that \eqref{eq:A4b} holds true. Then, using that $D_1(x^{-1})$ is increasing in $x$, we infer
\begin{eqnarray}
W_1
& \leq &
\varepsilon \limsup_{t \to \infty}\, \sum_{v \in \mathcal{T}_t} L(v) \int_{e^t}^{L(v)^{-1}} x^{-1} D_1(x^{-1}) \, \dx	\notag	\\
& \leq &
\varepsilon \limsup_{t \to \infty}\, \sum_{v \in \mathcal{T}_t} L(v) D_1(L(v)) \int_{e^t}^{L(v)^{-1}} x^{-1} \, \dx	\notag	\\
& \leq &
\varepsilon \limsup_{t \to \infty}\, \sum_{v \in \mathcal{T}_t} e^{-S(v)} D_1(e^{-S(v)}) (S(v)-t).
\label{eq:Nerman2}
\end{eqnarray}
At this point, we use a combination of Theorem 6.3 in \cite{Ner1981}/Theorem 7.1 in \cite{BK1997} and Theorem 7.2 in \cite{BK1997} that gives a Seneta-Heyde norming for Crump-Mode-Jagers processes counted with general characteristic (see the paragraph following Theorem 7.2 in \cite{BK1997}). Note that, according to the discussion in Section 8 of \cite{BK1997}, the function $L(x)$ in \cite[Theorem 7.2]{BK1997} can be chosen as $D_1(x)$. Thus,
\begin{equation*}
D_1(e^{-t}) \sum_{v \in \mathcal{T}_t} e^{-S(v)} (S(v)-t)
~=~	(1\!-\!\varphi(e^{-t})) \sum_{v \in \mathcal{T}_t} e^{-(S(v)-t)} (S(v)-t)
~\to~	c W
\end{equation*}
almost surely as $t \to \infty$ for some $c \in [0,\infty)$. An application of Theorem \ref{Thm:reverse_version} with $H(t) := D_1(e^{-t})$ (validity of conditions (ii) and (iii) in Theorem \ref{Thm:reverse_version} follows from the fact that $D_1$ is slowly varying at the origin and Theorems 1.2.1 and 1.5.6 in \cite{BGT1989}) yields
\begin{equation*}
\sum_{v \in \mathcal{T}_t} e^{-S(v)} D_1(e^{-S(v)}) (S(v)-t)	~\to~	cW
\end{equation*}
almost surely as $t \to \infty$. Using this in \eqref{eq:Nerman2}, we arrive again at $W_1 \leq \varepsilon c W$ almost surely and thus at a contradiction as in the previous case. The proof is herewith complete.
\end{proof}

\subsection{The proof of Theorem \ref{Thm:set_of_solutions_non-lattice}}	\label{subsec:Proofs_homogeneous}

We will derive Theorem \ref{Thm:set_of_solutions_non-lattice} from the following result that provides a complete description of the disintegrations of solutions to \eqref{eq:SumFP}.

\begin{Thm}	\label{Thm:1st_representation}
Suppose that \eqref{eq:A1}-\eqref{eq:A4} hold true and, furthermore, \eqref{eq:A5} if $\alpha = 1$. Then the disintegration $\Phi$ of any $\phi \in \Fsum$ has a representation of the form
\begin{equation}	\label{eq:1st_representation_continuous}
\Phi(t)	~=~
	\begin{cases}
	\exp\left(-\sigma^{\alpha} W |t|^\alpha
	\left[1-\i\beta \frac{t}{|t|}\tan\left(\frac{\pi\alpha}{2}\right)\right]\right),
	& \text{ if } \alpha \not \in \{1,2\}, \\
	\exp(\i \mu W t - \sigma W |t|),
	& \text{ if } \alpha = 1,	\\
	\exp(- \sigma^2 W t^2),
	& \text{ if } \alpha = 2,
	\end{cases}
\end{equation}
where $\mu \in \R$, $\sigma \geq 0$ and $\beta \in [-1,1]$.
\end{Thm}
\begin{proof}[Proof of Theorems \ref{Thm:set_of_solutions_non-lattice} and \ref{Thm:set_of_solutions_non-lattice_alpha=1} by Theorem \ref{Thm:1st_representation}]
Both theorems follow immediately from Theorem \ref{Thm:1st_representation} in combination with \eqref{eq:disintegration_integrated}.
\end{proof}

In order to prove Theorem \ref{Thm:1st_representation} we need to evaluate the random integral
\begin{equation*}
\int \left(e^{\i tx} - 1 - \frac{\i tx}{1+x^2} \right) \, \nu(\dx).
\end{equation*}
We know from Lemma \ref{Lem:Disintegration_evaluated} that $\nu|_{(0,\infty)}$ and $\nu|_{(-\infty,0)}$, the restrictions of $\nu$ to the positive and negative halfline, respectively, are random multiples of $x^{-(\alpha+1)}\dx$. The value of the integral can therefore be concluded from existing literature:

\begin{Lemma}	\label{Lem:int_continuous}
For any $t > 0$,
\begin{eqnarray*}
I_1(t)
& := &
\int_0^{\infty} \left(e^{\i tx} - 1 - \frac{\i tx}{1+x^2} \right) \, \frac{\dx}{x^{\alpha+1}}	\\
& = &
\begin{cases}
\i c t - t^{\alpha} e^{-\frac{\pi \i}{2}\alpha}\frac{1}{\alpha} \Gamma(1-\alpha)					&	\text{if } 0 < \alpha < 1,	\\
\i c t - (\pi/2) t - \i t \log t																		&	\text{if }	\alpha = 1,	\\
\i c t - \i t^{\alpha} e^{-\frac{\pi \i}{2}(\alpha-1)}\frac{\Gamma(2-\alpha)}{\alpha(\alpha-1)} 		&	\text{if } 1 < \alpha < 2,
\end{cases}
\end{eqnarray*}
where $\Gamma$ denotes Euler's Gamma function and $c$ is a real constant depending on the value of $\alpha$.
\end{Lemma}
\begin{proof}[Source]
See \textit{e.g.}\ \cite[pp.\,168]{GK1968}.
\end{proof}

\begin{proof}[Proof of Theorem \ref{Thm:1st_representation}]
Let $ \phi \in \Fsum$ with disintegration $\Phi$.
From Lemma \ref{Lem:Disintegration_evaluated} we infer that $\Phi = \exp(\Psi)$ with
\begin{equation}	\label{eq:2nd_shape_of_Psi}
\Psi(t)	~=~	\i W_1 t - \frac{W_2 t^2}{2}
+ \int \left(e^{\i tx} - 1 - \frac{\i tx}{1+x^2} \right) \, \nu(\dx),
\quad	t \in \R
\end{equation}

Suppose first that $\alpha = 2$. It then follows from Lemma \ref{Lem:Disintegration_evaluated} that $\nu = 0$ almost surely whence \eqref{eq:2nd_shape_of_Psi} simplifies to
\begin{equation*}
\Psi(t)	~=~	\i W_1 t - \frac{W_2 t^2}{2}
\quad	\text{almost surely,}
\end{equation*}
where $W_1, W_2$ are $\bL$ measurable. From \eqref{eq:Psi's_equation}, we infer that for all $t \geq 0$
\begin{equation}	\label{eq:Psi_alpha=2}
\i W_1 t - \frac{W_2 t^2}{2}
~=~ \i \sum_{|v|=n} L(v) [W_1]_vt -  \sum_{|v|=n} L(v)^2 [W_2]_v \frac{t^2}{2}
\quad	\text{almost surely.}
\end{equation}
By linear independence of $\i$\! and $1$, this yields that $W_1$ and $W_2$ are endogenous fixed points w.r.t.\ $T^{1}$ and $T^{2}$, respectively. Thus, $W_1=0$ almost surely by Theorem \ref{Thm:no_end_T^beta}. Since, furthermore, we know that $W_2 \geq 0$ almost surely, we obtain that $W_2 = 2 \sigma^2 W$ for some $\sigma \geq 0$ by Proposition \ref{Prop:uniqueness_of_end_FP}.

Now assume that $0 < \alpha < 2$. Then $W_2$ is still an endogenous fixed point w.r.t.\ $T^{2}$ by \eqref{eq:W_2}. On the other hand, $\alpha < 2$ implies that $W_2 = 0$ almost surely by Theorem \ref{Thm:no_end_T^beta}. We proceed with the evaluation of the integral in \eqref{eq:2nd_shape_of_Psi}. Recall that, by Lemma \ref{Lem:Disintegration_evaluated}, $\nu$ can be written as
\begin{equation*}
\nu(\dx)
~=~ W (c_1 x^{-(\alpha+1)} \1_{(0,\infty)}(x) \dx + c_2 |x|^{-(\alpha+1)} \1_{(-\infty,0)}(x) \dx)
\end{equation*}
for constants $c_1, c_2 \geq 0$ and the non-negative endogenous fixed point $W$ w.r.t.\ $T^{\alpha}$. Thus, for any $t > 0$,
\begin{eqnarray*}
\Psi(t)
& = &
\i W_1t + \int \left(e^{\i tx} - 1 - \frac{\i tx}{1+x^2} \right) \, \nu(\dx)	\\
& = &
\i W_1t + W(c_1 I_1(t) + c_2 \overline{I_1(t)}),
\end{eqnarray*}
where $\overline{I_1(t)}$ denotes the complex conjugate of $I_1(t)$.
If $c_1 = c_2 = 0$, then $\Psi(t) = \exp(\i W_1 t)$ almost surely and by Proposition \ref{Prop:Endogeny_disintegration}, $W_1$ is an endogenous fixed point w.r.t.\ $T$. Thus, $W_1=0$ almost surely, if $\alpha \not = 1$ by Theorem \ref{Thm:no_end_T^beta} and $W_1 = \mu W$ almost surely for some $\mu \in \R$ if $\alpha = 1$ by Theorem \ref{Thm:endogeneuous=>one_sided}. Therefore, let $c_1+c_2 > 0$ for the rest of the proof.
We will now apply Lemma \ref{Lem:int_continuous}:

In the case $0 < \alpha < 1$, this yields
\begin{eqnarray*}
\Psi(t)
& = &
\i W_1t + W(c_1 (\i ct - t^{\alpha} e^{-\frac{\pi \i}{2}\alpha}\Gamma(1-\alpha)/\alpha)	\\
& &	\hphantom{\i W_1t + W(}
+ c_2(-\i ct - t^{\alpha} e^{\frac{\pi \i}{2}\alpha}\Gamma(1-\alpha)/\alpha))	\\
& = &
\i (W_1 + c(c_1 - c_2)W )t - Wt^{\alpha}(c_1e^{-\frac{\pi \i}{2}\alpha} + c_2e^{\frac{\pi \i}{2}\alpha})\Gamma(1-\alpha)/\alpha)	\\
& = &
\i (W_1 + c(c_1 - c_2)W )t	\\
& &
-	\frac{\Gamma(1-\alpha)}{\alpha} Wt^{\alpha}
((c_1+c_2)\cos (\pi \alpha/2)-\i (c_1-c_2) \sin(\pi \alpha/2))	\\
& = &
\i \widetilde{W} t - \sigma^{\alpha} W |t|^\alpha
	\left[1-\i\beta \frac{t}{|t|}\tan\left(\frac{\pi\alpha}{2}\right)\right],
\end{eqnarray*}
where $\widetilde{W} := W_1 + c(c_1 - c_2)W$, $\sigma^{\alpha}:=\frac{\Gamma(1-\alpha)}{\alpha}(c_1+c_2)\cos (\pi \alpha/2) \geq 0$, and $\beta = (c_1-c_2)/(c_1+c_2) \in [-1,1]$. As in the case $\alpha = 2$, one can show via \eqref{eq:Psi's_equation} that $\widetilde{W}$ is an endogenous fixed point w.r.t.\ $T$. More precisely, \eqref{eq:Psi's_equation} implies that, almost surely for each $n \geq 0$,
\begin{align*}
\i \widetilde{W} t & - \sigma^{\alpha} W |t|^\alpha
	\left[1-\i\beta \frac{t}{|t|}\tan\left(\frac{\pi\alpha}{2}\right)\right]	\\
& =~
\i \sum_{|v|=n} L(v) [\widetilde{W}]_v t - \sigma^{\alpha} \sum_{|v|=n} L(v)^{\alpha} [W]_v |t|^\alpha
\left[1-\i\beta \frac{t}{|t|}\tan\left(\frac{\pi\alpha}{2}\right)\right].
\end{align*}
Dividing by $t$ and letting $t \to \infty$, we see that $\widetilde{W}$ is an endogenous fixed point w.r.t.\ $T$. Since $\alpha < 1$, $\widetilde{W}=0$ almost surely by Theorem \ref{Thm:no_end_T^beta}.

If $1 < \alpha < 2$, a similar argument as before leads to the desired conclusion.

Finally, assume $\alpha = 1$. Then an application of Lemma \ref{Lem:int_continuous} leads to
\begin{eqnarray*}
\Psi(t)
& = &
\i W_1t + W(c_1(\i ct -\frac{\pi}{2}t - \i t\log t) + c_2(-\i ct -\frac{\pi}{2}t + \i t\log t))	\\
& = &
\i (W_1 + Wc(c_1-c_2))t + W(c_1(-\frac{\pi}{2}t - \i t\log t) + c_2(-\frac{\pi}{2}t + \i t\log t))	\\
& = &
\i \widetilde{W} t - \sigma W t \left(1 + \i \beta \frac{2}{\pi} \log t \right),
\end{eqnarray*}
where $\widetilde{W} = W_1 + Wc(c_1-c_2)$, $\sigma = (c_1+c_2)\pi/2 > 0$, and $\beta = (c_1-c_2)/(c_1+c_2) \in [-1,1]$. Now use \eqref{eq:Psi's_equation} for $t=1$ to obtain that almost surely for any $n \geq 0$,
\begin{equation*}
\i \widetilde{W} - \sigma W	~=~	\sum_{|v|=n} L(v) [\widetilde{W}]_v - \sigma \sum_{|v|=n} L(v) [W]_v.
\end{equation*}
By linear independence of $\i$\! and $1$, $\widetilde{W}$ is an endogenous fixed point w.r.t.\ $T$. Therefore, $\widetilde{W} = \mu W$ almost surely for some $\mu \in \R$ by Theorem \ref{Thm:endogeneuous=>one_sided}. Again from \eqref{eq:Psi's_equation} but for $t>1$, we infer after some minor manipulations that
\begin{eqnarray*}
W \beta \log t 
& = &
\sum_{|v|=n} L(v) [W]_v \beta \log (L(v)t)	\\
& = &
\sum_{|v|=n} L(v) [W]_v \beta \log t + \sum_{|v|=n} L(v) \log (L(v)) [W]_v \beta	\\
& = &
W \beta \log t + \sum_{|v|=n} L(v) \log (L(v)) [W]_v \beta
\end{eqnarray*}
almost surely for each $n \geq 0$. Since $\sup_{|v|=n} L(v) \to 0$ almost surely by Lemma \ref{Lem:sup_L(v)_to_0}, we infer that $\sum_{|v|=n} L(v) \log (L(v)) [W]_v $ is ultimately strictly negative almost surely on $\Surv$, the set of survival of the supercritical weighted branching process. Thus, $\beta = 0$.
\end{proof}

\section{The inhomogeneous equation}	\label{sec:inhom}

We will solve the inhomogeneous equation by another use of disintegration. The strategy is to show that any disintegration of a $\phi \in \Fsum(C)$ can be decomposed into the product of the disintegration of some fixed solution $W^*$ to the inhomogeneous equation and the disintegration of a solution to the homogeneous equation. This approach is taken from \cite{AM2010}.

\subsection{Disintegration}	\label{subsec:disintegration_inhomogeneous}

For $\phi \in \Fsum(C)$, we define the corresponding multiplicative martingale by
\begin{equation} \label{eq:disintegrated_inhom}
\Phi_n(t)	~:=~	\phi_n(t,\bC \otimes \bT)	~:=~ \exp\Bigg(\i \sum_{|v|<n} L(v) C(v) t\Bigg) \cdot \prod_{|v|=n} \phi(L(v) t),
\quad n \geq 0.
\end{equation}
As in the homogeneous case, $(\Phi_n(t))_{n \geq 0}$ forms a martingale:

\begin{Lemma}	\label{Lem:Disintegration_inhom}
Let $\phi \in \Fsum(C)$ and $t \in \R$. Then $(\Phi_n(t))_{n \geq 0}$ forms a complex-valued bounded martingale with respect to $(\A_n)_{n \geq 0}$ and thus converges almost surely and in mean to a random variable $\Phi(t) = \Phi(t,\bC \otimes \bT)$ satisfying
\begin{equation}	\label{eq:disintegration_integrated_inhom}
\E \Phi(t) ~=~ \phi(t).
\end{equation}
\end{Lemma}
\begin{proof}
This can be proved in the same way as the corresponding result in the homogeneous case. We therefore omit supplying further details.
\end{proof}

In the inhomogeneous case, we define the random variables $W^*_n$, $n \geq 0$ by
\begin{equation}	\label{eq:W*_n}
W^*_n ~:=~	\sum_{|v| \leq n} L(v) C(v).
\end{equation}
By \eqref{eq:N<infty}, we have that $W^*_n$ is the sum of only finitely many non-zero terms almost surely and is therefore well-defined. It is natural to try to construct a fixed point of \eqref{eq:SumFP_inhom} by considering the limit of $W^*_n$ as $n \to \infty$. However, this limit need not exist. In what follows, we make the assumption that $W^*_n$ converges almost surely as $n \to \infty$:
\begin{equation}	\tag{A6}	\label{eq:W*}
W^*_n	\underset{n \to \infty}{\to} W^*
\text{ almost surely for some finite random variable } W^*.
\end{equation}
A number of sufficient conditions for \eqref{eq:W*} to hold will be provided in Subsection \ref{subsec:W*}. Of course, when $W^*$ exists as almost sure limit of $W_n^*$, then it is straightforward to check that it satisfies \eqref{eq:SumFP_inhom}. This observation is recorded in the following lemma:

\begin{Lemma}
If \eqref{eq:W*} holds, then $W^*$ defines a solution to \eqref{eq:SumFP_inhom}
\end{Lemma}
%\begin{proof}
%Under \eqref{eq:W*},
%\begin{eqnarray*}
%W^*
%& = &
%\lim_{n \to \infty} \sum_{|v| \leq n} L(v) C(v)
%~=~	C(\varnothing) + \lim_{n \to \infty} \sum_{j=1}^N T_j \sum_{|v|\leq n-1} [L(v)]_j C(jv)	\\
%& = &
%C(\varnothing) + \sum_{j=1}^N T_j \lim_{n \to \infty} [W^*_{n-1}]_j
%~=~	C(\varnothing) + \sum_{j=1}^N T_j [W^*]_j
%\quad	\text{almost surely,}
%\end{eqnarray*}
%where we have utilized that $N<\infty$ almost surely by \eqref{eq:N<infty}.
%\end{proof}

\begin{Prop}	\label{Prop:Phi=exp(-W*)xPhi_{hom}}
Assume that \eqref{eq:W*} holds and let $\phi \in \Fsum(C)$ with disintegration $\Phi$. Then
\begin{equation}	\label{eq:Phi=exp(-W*)xPhi_{hom}}
\Phi(t)	~=~	\exp(\i W^* t) \, \Phi_{\mathrm{hom}}(t)	\quad	\text{almost surely}	\quad	(t \in \R)
\end{equation}
where $\Phi_{\hom}$ denotes the disintegration of a (not necessarily non-trivial) solution to \eqref{eq:SumFP}. Conversely, any characteristic function $\phi$ obtained by taking the expectation of a process $\Phi$ as in \eqref{eq:Phi=exp(-W*)xPhi_{hom}} defines a solution to \eqref{eq:SumFE_inhom}.
\end{Prop}
This result can be proved along the lines of the proof of Theorem 4.4 in \cite{AM2010} but using characteristic functions instead of Laplace transforms. Therefore, we refrain from giving further details.

\begin{proof}[Proof of Theorem \ref{Thm:set_of_solutions_inhom_non-lattice}]
The result follows immediately from Propositions \ref{Prop:Phi=exp(-W*)xPhi_{hom}} and \ref{Prop:W*} in combination with the main results in the homogeneous case.
\end{proof}

\subsection{Sufficient conditions for $W^*$ to be well-defined}	\label{subsec:W*}

\begin{Prop}	\label{Prop:W*}
Assume that \eqref{eq:A1}-\eqref{eq:A4} hold true. Then each of the following conditions is sufficient for \eqref{eq:W*} to hold:
\begin{itemize}
	\item[(i)]
		$m(1) < \infty$, $\E |C| < \infty$,
		and $(\Smoothn(\delta_0))_{n \geq 0}$ is $\mathcal{L}^p$-bounded for some $p \geq 1$.
	\item[(ii)]
		$m(\beta) < 1$ and $\E |C|^{\beta} < \infty$ for some $0 < \beta \leq 1$.
\end{itemize}
\end{Prop}

Before we present the proof of Proposition \ref{Prop:W*}, we give an auxiliary result.

\begin{Lemma}	\label{Lem:W*_n}
Suppose that \eqref{eq:A1}-\eqref{eq:A4} hold. Further, assume that $m(1) < \infty$ and $\E |C| < \infty$. Then
\begin{equation*}
(W^*_n)_{n \geq 0} \text{ is a }
	\begin{cases}
	supermartingale	&	\hphantom{\text{w.r.t.\ } (\A_{n+1})_{n \geq 0} \text{\ \ iff\ \ }}		\E C \leq 0;	\\
	martingale			&	\text{w.r.t.\ } (\A_{n+1})_{n \geq 0} \text{\ \ iff\ \ }					\E C = 0;	\\
	submartingale		&	\hphantom{\text{w.r.t.\ } (\A_{n+1})_{n \geq 0} \text{\ \ iff\ \ }}		\E C \geq 0.
	\end{cases}
\end{equation*}
\end{Lemma}
\begin{proof}
Since $W^*_n-W^*_{n-1} = \sum_{|v|=n} L(v) C(v)$ for each $n \geq 1$, we find that 
\begin{equation*}
\E[W^*_n-W^*_{n-1} \,|\, \A_n]
~=~	(\E C) \, \sum_{|v|=n} L(v)
\quad	\text{almost surely}
\end{equation*}
when taking into account that the $L(v)$, $|v|=n$ are $\A_n$-measurable and the $C(v)$, $|v|=n$ are independent of $\A_n$.
\end{proof}

\begin{proof}[Proof of Proposition \ref{Prop:W*}]
If (i) holds, we infer from Lemma \ref{Lem:W*_n} that $(W^*_n)_{n \geq 0}$ is a (super-,sub-) martingale. Since it is $\mathcal{L}^p$-bounded by (i), an application of the martingale convergence theorem yields the almost sure convergence.
%of $(W^*_n)_{n \geq 0}$.

\noindent
(ii) follows from the estimate
\begin{equation*}
\E|W_n^*|^{\beta}	~\leq~	\E \sum_{|v|\leq n} L(v)^{\beta} |C(v)|^{\beta}
~\underset{n \to \infty}{\to}~	\frac{\E |C|^{\beta}}{1-m(\beta)}.
\end{equation*}
\end{proof}

\subsection{The fixed points of the \texttt{Quicksort} mapping}

Recall the \texttt{Quicksort} equation \eqref{eq:Quicksort} from Subsection \ref{subsec:inhom_results}:
\begin{equation*}
X	~\stackrel{\mathrm{d}}{=}~	U X_1 + (1-U) X_2 + g(U)
\end{equation*}
where $U \sim \mathrm{Unif}(0,1)$, $X_1,X_2$ are i.i.d.\ copies of $X$ independent of $U$, and
$g:(0,1) \to (0,1)$, $u \mapsto 2u \log u + 2 (1-u) \log(1-u) + 1$.
We now derive Corollary \ref{Cor:Quicksort} from Theorem \ref{Thm:set_of_solutions_inhom_non-lattice}. To this end, notice that in the given context $N = 2$ and
\begin{equation*}
m(\theta)	~=~	\E (U^{\theta} + (1-U)^{\theta}) ~=~ 2 \E U^{\theta} ~=~ \frac{2}{1+\theta},
\quad	\theta \geq 0.
\end{equation*}
Thus, assumptions \eqref{eq:A1}-\eqref{eq:A5} are fulfilled with $\alpha = 1$. Moreover, by induction, $\sum_{|v|=n} L(v) = 1$ for all $n \geq 0$, so that $W=1$ is (up to scaling) the unique positive endogenous fixed point of \eqref{eq:SumFP}.
Further, a straightforward calculation gives
\begin{equation*}
\E |W^*_n - W^*_{n-1}|^2
~=~
\E \Bigg(\sum_{|v|=n} L(v) C(v) \Bigg)^2
~=~	(2/3)^2 \E C^2
\end{equation*}
for all $n \geq 0$.
%\begin{eqnarray*}
%\E |W^*_n - W^*_{n-1}|^2
%& = &
%\E \Bigg(\sum_{|v|=n} L(v) C(v) \Bigg)^2	\\
%& = &
%\E \sum_{|v|=n} \sum_{|w|=n} L(v) C(v) L(w) C(w).
%\end{eqnarray*}
%Moreover, for any $v \in \V$ with $|v|=n$, $L(v)$ is the product of $n$ independent $\A_n$-measurable %uniform $(0,1)$-variables and, therefore, strictly less than $1$ with probability. Further $C(v)$ is a %bounded mean-zero random variable independent of $\A_n$. Thus,
%\begin{eqnarray*}
%\E L(v) L(w) C(v) C(w) ~=~ \E (L(v)L(w) \E[C(v) C(w) | \A_n]) ~=~ 0
%\end{eqnarray*}
%for any $v,w \in \N^n$, $v \not = w$. Consequently, we have
%\begin{eqnarray*}
%\E |W^*_n - W^*_{n-1}|^2
%& = &
%\E \sum_{|v|=n} L(v)^2 C(v)^2	\\
%& = &
%2^n \E L(1\ldots1)^2 C(1\ldots1)^2	\\
%& = &
%2^n (\E U^2)^n \E C^2
%~=~	(2/3)^n \E C^2.
%\end{eqnarray*}
Hence, $(W_n^*)_{n \geq 0}$ is an $\mathcal{L}^2$-bounded martingale.
The assumptions of Theorem \ref{Thm:set_of_solutions_inhom_non-lattice} are thus fulfilled and we conclude that any characteristic function $\phi$ of a solution to \eqref{eq:Quicksort} is of the form
\begin{equation*}
\phi(t)	~=~	\E \exp\left(\i W^* t + \i \mu W t -\sigma W|t|\right)
				~=~	\phi^*(t) \, e^{\i \mu t -\sigma |t|}
\end{equation*}
where $\phi^*$ denotes the characteristic function of $W^* = \lim_{n \to \infty} \sum_{|v| \leq n} L(v)C(v)$. As the limit of a mean-zero $\mathcal{L}^2$-bounded martingale, $W^*$ has zero mean and finite variance. On the other hand, if $C_n$ denotes the number of key comparisons \texttt{Quicksort} requires to sort a list of $n$ distinct numbers, then $P$, the distributional limit of $(C_n-\E C_n)/n$ as $n \to \infty$, is known \cite{Roe1991} to be the unique solution to \eqref{eq:Quicksort} with zero mean and finite variance. Thus, $W^*$ has distribution $P$ and the proof of Corollary \ref{Cor:Quicksort} is complete.

\section*{Acknowledgements}
We would like to thank Svante Janson for valuable discussions in the course of the project.
We are further indebted to two anonymous referees for their constructive remarks that helped improving the presentation of the paper.

\end{document}